\documentclass{article}
\usepackage{amsmath}    
\usepackage{amsthm}     
\usepackage{amssymb} 
\usepackage{mathtools}
\usepackage{cancel}
\usepackage{todonotes}
\definecolor{darkblue}{rgb}{0, 0, 0.6}

\usepackage{graphicx}   
\usepackage{xcolor}      
\usepackage
[colorlinks=true, pdfstartview=FitV, linkcolor=darkblue, citecolor=darkblue, urlcolor=blue]
{hyperref}
\usepackage[font=small,labelfont=bf]{caption}
\def \CSA{\mathfrak h}
\def \bs {\boldsymbol}

\def\pa{\partial}
\def \U#1{  \le[\begin{array}{cc} 1& #1\\
0 &1
\end{array}
\ri]}
\def \L#1{  \le[\begin{array}{cc} 1& 0\\
#1 &1
\end{array}
\ri]}

\def \wt{\widetilde}

\def \CM{\mathbb G}
\def\R{\mathbb R}
\def \S{\mathfrak S}
\def\res{\mathop{\mathrm {res}}}
\def \s{\sigma}
\def \l{\lambda}
\def \Id{\mathbf 1}
\def \QED{\hfill $\blacksquare$\par\vskip 4pt}
\def \d{\mathrm d}
\usepackage[all]{xy}
\def \le{\left}
\def\nn{\nonumber}
\def \ri{\right}
\usepackage{pgf,tikz,pgfplots}
\usepackage{mathrsfs}
\usetikzlibrary{arrows,decorations.markings}
\usetikzlibrary{calc}
\usetikzlibrary{shapes.geometric,positioning}
\def\C{\mathbb C}
\def \wh {\widehat}
\theoremstyle{plain} 
\newtheorem{thm}{Theorem}[section] 
 
\newtheorem{lem}[thm]{Lemma} 
\newtheorem{prop}[thm]{Proposition}

\newtheorem{defn}[thm]{Definition}

\newtheorem{rmk}[thm]{Remark}

\DeclareMathOperator{\Tr}{Tr}

\def \be#1\ee{\begin{align}#1\end{align}}
\def \tr{\mathrm {Tr}}

\tikzstyle directed=[postaction={decorate,decoration={markings,
		mark=at position .65 with {\arrow{latex}}}}]

\renewcommand{\theequation}{\arabic{section}.\arabic{equation}}

\makeatletter
\@addtoreset{equation}{section}
\makeatother

\begin{document}


\begin{center}
\begin{huge}
{\bf  Stokes manifolds and cluster algebras}
\end{huge}\\
\bigskip
M. Bertola$^{\dagger\ddagger}$\footnote{Marco.Bertola@\{concordia.ca, sissa.it\}},  
S. Tarricone$^{\dagger\diamondsuit}$ \footnote{tarricone@math.univ-angers.fr, Sofia.Tarricone@concordia.ca},
\\
\bigskip
\begin{small}
$^{\dagger}$ {\it   Department of Mathematics and
Statistics, Concordia University\\ 1455 de Maisonneuve W., Montr\'eal, Qu\'ebec,
Canada H3G 1M8} \\
\smallskip
$^{\ddagger}$ {\it  SISSA/ISAS,  Area of Mathematics\\ via Bonomea 265, 34136 Trieste, Italy }\\
$^{\diamondsuit}$ LAREMA, UMR 6093, UNIV Angers, CNRS, SFR Math-Stic, France
\end{small}
\vspace{0.5cm}
\end{center}

 \begin{center} \bf Abstract 
 \end{center}
Stokes' manifolds, also known as wild character varieties, carry a natural symplectic structure. 
Our goal is to provide explicit log-canonical coordinates for these natural Poisson structures on the  Stokes' manifolds of polynomial connections of rank $2$, thus including the second Painlev\'e\ hierarchy. This construction provides the explicit linearization of the Poisson structure first discovered by Flaschka and Newell  and then rediscovered and generalized by Boalch. 
We show that, for a connection of degree $K$, the Stokes' manifold is a cluster manifold of type $A_{2K}$. The main idea is then applied to express explicitly also the log--canonical coordinates for the Poisson bracket introduced by Ugaglia in the context of Frobenius manifolds and then also applied by Bondal in the study of the symplectic groupoid of quadratic forms.

\tableofcontents
\section{Introduction and results}
Rational connections on the Riemann sphere (and more general Riemann surfaces) are intimately connected with the theory of integrable systems in general, and Painlev\'e\ equations in particular (or Hitchin's systems). 

The (generalized) monodromy map which associates the monodromy, Stokes' and connection matrices to a rational connection, has been shown to provide integrals of motion for all the Painlev\'e\ equations (and hence parametrize their solutions) and the generalizations thereof that culminated in the eighties with the work of the Japanese school \cite{JMU1, JMU2, JMU3}. 

The Painlev\'e\ equations themselves can be cast as Hamiltonian systems \cite{Okamoto}; the underlying Poisson structure coincides with the standard Poisson structure on coadjoint orbits of suitably defined loop groups \cite{Adams}.

The (extended) monodromy map provides a connection between these ODE (rational connections) and the representations theory of the fundamental group of the punctured sphere, which are called ``character varieties'' (for the case of Fuchsian singularities) or certain generalizations thereof that go under the picturesque name of ``wild character varieties'' \cite{wildcharacter}.\\[1pt]

The natural question, from the point of view of symplectic geometry, is to identify the push--forward of the Poisson structure from the space of coadjoint orbits to the space of extended monodromy data. 

Possibly the first paper to address this issue was \cite{FlaschkaNewell} in the eighties; they considered the prototypical example of ``wild'' character variety, namely, the Stokes' phenomenon of a rank-two differential equation with polynomial coefficients  , which underlies the Painlev\'e\ II hierarchy. They provided the explicit expressions for the Poisson brackets between the Stokes' parameters which are the image, under the generalized monodromy map, of the  Lie-Poisson structure on the matrix of the differential equation. 

On a similar track, but for the regular monodromy map, in the work  \cite{KorSam} it was first realized that the Lie-Poisson structure on a Fuchsian differential connection induces the Goldman Poisson structure \cite{Goldman}  on the matrices of the monodromy representation.

The seminal paper of Flaschka and Newell was largely ignored for almost twenty years and then its idea applied to a different type of rational connections by Ugaglia \cite{Ugaglia} in relationship with the theory of Frobenius manifolds; here the connection has one Fuchsian and one irregular singularity of Poincar\'e\ rank 2 at infinity.

The topic was then put into the framework of quasi--Hamiltonian geometry in a series of papers by Boalch \cite{Boalch1,smapg, BoalchDuke} in the beginning of the millennium, where he re-derived the symplectic and Poisson geometry induced on the space of generalized monodromy data.

A practical issue that has not been given much consideration until lately, is that of providing explicit Darboux coordinates (or, rather, log-canonical) for these structures; 
in the last decade the works of Fock-Goncharov \cite{FG} have shown the connection between the geometry of character varieties and the emergent study of cluster algebras. 
Their setup is very general but phrased in a setting which is abstracted from the direct relationship with the theory of the monodromy map.

In the recent \cite{Bertola:2019ws}  the Fock-Goncharov formalism has been shown to provide explicit log-canonical coordinates for the Goldman Poisson structure on the character variety of an arbitrary punctured Riemann surface.

The relationship of the Fock--Goncharov coordinates with the Stokes' phenomenon of second order ODEs on Riemann surfaces was pointed out in \cite{Nakanishi}; however this relationship is for a different Stokes' phenomenon, corresponding to the Wentzel-Kramers-Brillouin asymptotic expansion with respect singularly perturbed ODEs. 

The present paper instead considers the classical Stokes' phenomenon for (polynomial) ODEs and provides explicit log-canonical (and Darboux) coordinates for the Poisson structure on the simplest class of such wild character varieties. 
We thus establish the main result that can be summarized  in the following theorem \\[0pt]

\noindent {\bf Theorem.}\ \   {\it
The wild character variety of an $sl_2$ polynomial connection of degree $K$  on the Riemann sphere is a cluster manifold of type $A_{2K}$. The log--canonical Poisson (symplectic) structure on this cluster variety coincides with the push--forward by the monodromy map of the Lie-Poisson structure.
}

\paragraph{Detailed description of results.}
In Section \ref{general} we provide a self--contained explanation of the symplectic structure on Stokes' matrices associated to an arbitrary {\it polynomial} differential equation in $sl_n$; the result (in even greater generality) was derived by Boalch \cite{BoalchDuke}, generalizing the result of \cite{FlaschkaNewell} and \cite{Ugaglia}. The proof we provide, however, is completely different from loc. cit. and uses ideas developed around the notion of the {\it Malgrange one-form}  in \cite{BertolaIsoTau, BertolaCorrection}.

We then consider in detail the case of $sl_2$ polynomial differential equation of arbitrary Poincar\'e\ rank $K+1$.  The goal is to provide explicit parametrizations for the Stokes' parameters and to show that  the coordinate charts thus defined glue together in the fashion of a cluster manifold of type $A_{2K}$. 
We recall that the Stokes' manifold in this case consists in the following algebraic variety
\begin{equation}\label{stokesmanifoldintro}
\mathfrak{S}_{K}=\left\lbrace 
\begin{pmatrix}
1&s_{1}\\
0&1
\end{pmatrix} 
\begin{pmatrix}
1&0\\
s_{2}&1
\end{pmatrix} \dots
\begin{pmatrix}
1&s_{2K+1}\\
0&1
\end{pmatrix}
\begin{pmatrix}
1&0\\
s_{2K+2}&1
\end{pmatrix} \lambda^{\s_3} = \Id \;\ \  \text{with} \; s_i \in \mathbb{C} , \ \ \l\in \C^\times
\right\rbrace
\end{equation} 
of complex dimension $2K$ (here and below $\s_{1,2,3}$ denote the three Pauli matrices).

The results of \cite{FlaschkaNewell} can be summarized in the following Poisson brackets (see Def. \ref{FNPB}): 
\be
\label{PB_K} 
\Big\{s_j,s_l \Big\}_{_{FN}} & =  \delta_{j,l-1} - \frac{\delta_{j,1}\delta_{l, 2K+2}} {\lambda^{2}}  + (-1)^{j-l+1} s_j s_l, \qquad j<l.\nn\\
\Big\{ s_j,\lambda \Big\}_{_{FN}} &=(-1)^{j}  s_j\lambda.
\ee
It is interesting  here to observe that the brackets \eqref{PB_K} satisfy the Jacobi identity in fact on the manifold of dimension $2K+3$  consisting of the parameters $s_j$'s and $\lambda$ without any restriction, where it has a single Casimir function given by the trace of the product in the left hand side of the equality in   \eqref{stokesmanifoldintro}. The variety $\mathfrak S_K$ is a Poisson submanifold and the bracket becomes nondegenerate (see Prop. \ref{propideal})
The explicit parametrization of the Stokes' data on the constrained manifold \eqref{stokesmanifoldintro} is given by (see Lemma \ref{prop:1})
\begin{align*}
&s_1 = -y_1 ^{-2}\\ 
&s_{2k} = (1+y_{2k}^2) \prod_{1\leq j\leq 2k} y_{j}^{(-1)^{j+1}2},\;\; k=1,\dots , K \\
& s_{2k+1} = -(1+y_{2k+1}^2) \prod_{1\leq j\leq 2k+1} y_{j}^{(-1)^j 2}, \;\; k=1,\dots , K-1\\
&s_{2K+1} = - \prod_{1\leq j\leq 2K} y_{j}^{(-1)^j 2} ,\\
&s_{2K+2} = y_1^{2}\left( 1+y_2^2\left( \dots \left(1+y_{2K}^2\right) \dots \right) \right) \prod_{j=1}^{K} y_{2j}^{-4}, \\
&\lambda = \prod_{j=1}^{K}y_{2j}^2 .
\end{align*}
and these coordinates $y_1,\dots, y_{2K}$ are log--canonical; the matrix of their Poisson brackets is  (Lemma \ref{lemmaPK})
\be
\mathbf{P}_{K} = \frac{1}{4}
\begin{pmatrix}
0 & 1 & 0 &0 & 0 & \dots &0\\
-1&0&1 &0 & 0 & \dots &0 \\
0&-1&0&1&0&\dots&0\\
\vdots & &\ddots & \ddots & \ddots& &\vdots\\
\vdots & &&\ddots & \ddots & \ddots& \vdots\\
0&0&\dots & & -1&0&1\\
0&0&\dots & & 0 & -1 &0
\end{pmatrix}
\ee 
The particular parametrization described here corresponds to a certain triangulation (see Fig. \ref{fig:triang and dynk})  of the $2(K+1)$ regular polygon, much in the same way as the relationship between the Grassmannian of $2$-planes (\cite{GSV}, Ch. II). The reader versed in the theory of cluster algebras will then recognize in the matrix above the matrix representing the simple quiver of type $A_{2K}$; this means that the variables $y_j^2$ form a {\it seed} for the cluster algebra of type $A_{2K}$. 

To complete the picture we need to show that different choices of triangulations of the regular $2K+2$--gon yield parametrizations of the Stokes' data that are obtained from the initial seed by applying a suitable sequence of {\it mutations }, i.e. simple birational maps from one chart to another.  This is accomplished in Sec. \ref{sec3} and specifically Sec. \ref{secflipping}.
We prove (Thm. \ref{thmmain}) that the Flaschka--Newell Poisson bracket coincides with the described above.

To conclude, in Sec. \ref{secUgaglia} we provide, using the same ideas used in the main text, the explicit log-canonical coordinates for the Poisson bracket introduced in \cite{Ugaglia}. This construction provides an alternative approach to the recent work by Checkov and Shapiro \cite{Chekhov}. 

\section{Symplectic structure on Stokes' matrices}
\label{general}
Consider a polynomial ODE of the form 
\be
\label{ODE}
\Psi'(z) = A(z)\Psi(z),\ \ \ A(z):= \sum_{j=1}^K A_j z^j.
\ee
For the sake of this discussion we can consider the case of $n\times n$ matrices (without real loss of generality, we consider the $sl_n$ case with $\tr(A(z))\equiv 0$). It is a technical but simple exercise to extend all the results to an arbitrary semisimple Lie algebra.
We assume that $A_K$ has simple eigenvalues (i.e. it is regular semisimple). Under this hypothesis one can find  a  solution in the class of formal series  of the form 
\be
\label{Psiform}
\Psi_{form}(z) = \wh Y(z) z^{-L} {\rm e}^{T(z)}, \ \ \ \wh Y(z) := G_0\bigg( \Id + \sum_{j\geq 1} \frac{Y_j}{z^j} \bigg)\in SL_n[[z^{-1}]], 
\ee
where $G_0$ is a chosen diagonalizing matrix for $A_K$ and $L, T(z)$ are diagonal traceless matrices. The entries of $L$ are called the {\it exponents of formal monodromy} and the matrix $T$ is a polynomial of the form 
\be
\label{Birkhof}
T(z) = T_{K+1} \frac {z^{K+1}}{K+1} + \dots + T_1 z,  \ \ T_j \in \CSA,
\ee
where $\CSA$ denotes the Cartan subalgebra of $sl_n$, namely diagonal traceless matrices. The coefficients of $T(z)$ are the (higher formal) {\it Birkhoff invariants}. The matrix $T_{K+1}$ is the diagonal form of the leading coefficient $A_K$, so that 
\be
A_{K}= G_0 T_{K+1} G_0^{-1}. 
\ee

\paragraph{Poisson structure on the space of matrices $A(z)$.}

The Lie Poisson structure on the set of rational matrices can be expressed as (for a  review see  \cite{Babelon})
\be
\label{KKS}
\{A(z) \mathop{,}^\otimes A(w)\} = \bigg[\frac {\Pi}{z-w}, \mathop {A}^1(z) + \mathop {A}^2(w)\bigg]
\ee
where
 $\mathop A^1(z):= A(z)\otimes \Id$, $\mathop A^2(w):= 
\Id \otimes A(w)$
and  $\Pi: \C^n\otimes \C^n\to \C^n\otimes \C^n$ is the tensor effecting the flip:
\be
\Pi (v\otimes f) = f\otimes v, \  \ v, f\in \C^n.
\ee
It can be explicitly written as $\Pi = \sum_{k,j=1}^n \mathbb E_{k,j} \otimes \mathbb E_{j,k}$, with $\mathbb E_{ij}$ the elementary matrices. In our case $A(z)$ is a polynomial; the matrix $A_{K}$ is easily seen to consist entirely of Casimir functions for this Poisson structure. 
The symplectic leaves are thus described; let $G(z)$ be the matrix of eigenvectors for $A(z)$ of the form 
\be
G(z) = G_0\bigg(\Id + \sum_{j\geq 1} \frac {B_j}{z^j}\bigg).
\ee
(The Laurent series has a finite radius of convergence).
Then 
\be
\label{defD}
A(z) = G(z) D(z) G(z)^{-1},\ \ D(z) = T_{K+1} z^K + \dots + T_1 - \frac {L } z + \dots
\ee
where the matrices $T_j$ are all diagonal traceless matrices; as the choice of letters suggests, they coincide (a simple exercise) with the Birkhoff invariants and the exponents of formal monodromy, while the rest of the Laurent tail plays no role in our present considerations.
Then the Casimir  functions are $T_1,\dots, T_{K+1}$ and $A_K = G_0 T_{K+1} G_0^{-1}$. (see \cite{Babelon}, Ch. III). 

On the symplectic leaves, the Poisson structure \eqref{KKS} has the form of the  ``universal symplectic structure'' of Krichever and Phong \cite{Krichever, KricheverPhong}:
\be
\label{kricphong}
\omega_{_{KK}} &=-\res_{z=\infty} \tr \bigg( D(z)  G(z)^{-1} \delta G(z)\wedge  G(z)^{-1} \delta G(z) \bigg)\d z\cr
&= -\res_{z=\infty} \tr \bigg( A(z)  \delta G(z) G(z)^{-1}  \wedge  \delta G(z )  G(z)^{-1}\bigg)\d z
\ee
The two form is invariant under gauge action of  right multiplication of $G$ by diagonal matrices of the form
\be
F(z) = \Id + \sum_{j\geq 1} \frac {F_j}{z^j}\ \ \in\ \  \CSA [[z^{-1}]].
\ee
To see this we introduce the symplectic potential 
\be
\theta:= \res_{z=\infty} \tr  \bigg( D(z)  G(z)^{-1} \delta G(z)\bigg)
\label{potKKS}
\ee
which has the property that $\delta \theta = \omega_{_{KK}}$. 
Now observe that under the gauge transformation $G(z)\mapsto G(z) F(z)$ we have
\be
\theta \mapsto \theta + \res_{z=\infty} \tr  \bigg( D(z) F^{-1}(z)  \delta F(z) \bigg)\d z. \label{gauge}
\ee
In the latter term, since $F(z) = \Id + \mathcal O(z^{-1})$ only the non-negative powers of $D(z)$ contribute (since $F^{-1}(z) \delta F(z) = \mathcal O(z^{-1})$). Given that the parameters $T_1,\dots, T_{K+1}$ in \eqref{defD} are constants, we can express the last term   in \eqref{gauge} as the total derivative of the function 
\be
 \res_{z=\infty} \tr  \bigg( D(z) F^{-1}(z)  \delta F(z) \bigg)\d z =  \delta \res_{z=\infty} \tr  \bigg( D(z)\ln F(z) \bigg)\d z, 
\ee
which implies that $\omega_{_{KK}}  = \delta \theta$ is indeed invariant. 
It is also invariant under left multiplication  $G(z) \mapsto H G(z)$ with $H$ a constant (in $z$): indeed, the left multiplication by a constant matrix $H$ leaves $\theta$ completely invariant:
\be
\theta\mapsto \theta +  \res_{z=\infty} \tr  \bigg(G(z) D(z) G^{-1}(z) H^{-1}\delta H \bigg)\d z =\theta
\ee
where we have used that $G(z) D(z) G^{-1}(z) = A(z)$ is a polynomial. 

The core of the idea of the ``extended coadjoint orbit'' of \cite{BoalchDuke} is the following: while $A_K = G_0 T_{K+1} G_0^{-1}$ is a Casimir for the KKS symplectic structure, $G_0$ itself is not because  right multiplications by a {\it constant} diagonal matrix do not leave the symplectic form invariant. 

Thus we allow $G_0$ to be kinematical variables: fix the Birkhoff invariants $T(z) = \sum_{j=1}^{K+1} T_{j} z^j/j$ (i.e. the diagonal traceless matrices $T_1,\dots, T_{K+1}$) and consider the set
\be
\wh{\mathcal O}_T:= \bigg\{(G_0, A(z))\in SL_n\times \mathcal A_K:\ \  G_0^{-1} A_K G_0 = T_{K+1},\ \ \  (G(z)^{-1} A(z) G(z))_+ = T'(z)\bigg\}\label{extOrbit},
\ee
where $( )_+$ denotes the Taylor part of a Laurent series (here is a polynomial part).

The dimension of $\wh {\mathcal O}_T$ is 
\be
\dim_\C\le(\wh{\mathcal O}_T\ri) = (K+1)(n^2-1) + (n-1) - (K+1)(n-1) = Kn(n-1) + n^2-1
\ee
The extended orbit $\wh{\mathcal O}_T$ carries the following $SL_n$--action:
\be
(G_0, A(z)) \mapsto (H G_0, H A(z) H^{-1}),\ \ \  H\in SL_n. 
\ee
Then the quotient $\wh{\mathcal O_T} / SL_n$ is a symplectic manifold of dimension $Kn(n-1)= \dim_{\C} \mathfrak S_K$. \\[12pt]

In order to connect the Lie--Poisson structure  with the Flaschka--Newell structure on the Stokes' matrices we need first a lemma and to describe the Stokes' phenomenon.

\begin{lem}
\label{lemmaeig}
The first $K+1$  coefficient matrices $Y_1,\dots, Y_{K+1}$ in the expansion of the formal solution $\Psi_{form}$ \eqref {Psiform}
coincide with the expansion of the eigenvector matrix, to wit 
\be
\wh Y(z) := G_0\le(\Id + \sum_{j\geq 1} \frac {Y_j} {z^j}\ri) = G(z) + \mathcal O(z^{-K-2}).
\ee
\end{lem}
\noindent
{\bf Proof.}
The formal series $\wh Y$ satisfies the ODE
\be
\wh Y'(z) + \wh Y(z)\le(T'(z) - \frac L z\ri) = A(z) \wh Y(z).
\ee
Since $\wh Y'(z) = \mathcal O(z^{-2})$, the matrices $T(z), L$ are diagonal  and since the degree of $A$ is $K$ we deduce that 
$\wh Y$ matches the Laurent expansion of the eigenvector matrix $G(z)$ up to the indicated order.
\QED

\paragraph{Description of the Stokes' phenomenon (extended monodromy map).}
The plane can be partitioned into $2K+2$ sectors of equal angular width $\mathcal S_\mu$, arranged  in counterclockwise order
; within each such sector, there exists a unique analytic solution $\Psi_\mu(z)$ to the ODE \ref{ODE} such that \cite{Wasow}
\be
\Psi_\mu(z) \simeq \Psi_{form}(z), \ \ \ |z|\to\infty, \ \ \arg z\in \mathcal S_\mu,
\ee
with $\Psi_{form}$ given in \eqref{Psiform}.
In these asymptotics, the determination of the matrix of formal exponents $z^L$ is the same, --say-- the principal one.
The matrix $S_\mu:=\Psi_\mu^{-1}(z) \Psi_{\mu+1}(z) $ is a constant (in $z$) matrix and it is called the {\bf Stokes'} matrix; if the entries $t_1,\dots, t_n$ of $T_{K+1}$ are  arranged in  increasing order of $\Re(t_j {\rm e}^{\theta_0})$ (for a generic $\theta_0$ so that this order is unique), then  the Stokes' matrices are all triangular matrices with unit diagonal, namely they belong to $N_\pm\subset SL_n$. Specifically, they alternate the triangularity as we move counterclockwise.

The entries of these matrices are not independent; they must satisfy the monodromy relation 
\be
S_1S_2\cdots S_{2K+2} {\rm e}^{2i\pi L} = \Id\label{nomonodromy}
\ee
which is a  consequence of the fact that the ODE has no singularities in the finite part of the plane and therefore each  of the solutions $\Psi_\mu$ extends uniquely to an entire matrix--valued function.
We thus define the Stokes' manifold as the set of these data:
\begin{defn}
The Stokes' manifold is the following set 
\be
\label{StokesMfld}
\mathfrak S_K:= \bigg\{ (S_1,\dots, S_{2K+2}, L)\in (N_+\times N_-)^{K+1}\times \CSA:\ \ 
S_1\cdots S_{2K+2} {\rm e}^{2i\pi L} = \Id. \bigg\}
\ee
where $N_\pm$ denote the solvable subgroups of upper/lower triangular matrices with ones on the diagonal and $\CSA$ denotes the subalgebra of diagonal traceless matrices.
The dimension of this manifold is 
\be \label{eq: general n dim sigma K}
\dim_\C (\mathfrak S_K )= Kn(n-1).
\ee
\end{defn}
It is apparent that the dimension is even; in fact Boalch \cite{BoalchDuke} shows that these type of manifolds are symplectic. We are going to give a self--contained description, adapted to this case, of this structure. This description deviates, in minor ways, from loc. cit.

\paragraph{The Malgrange form associated to an analytic  family of Riemann Hilbert problems.}
We describe here the gist of \cite{BertolaIsoTau, BertolaCorrection}. 
Suppose that $\Sigma\subset \C$ is a collection of oriented smooth arcs (intersecting transversally) and $J:\Sigma \to SL_n$ a smooth matrix--valued function (the ``jump matrix'') depending analytically on parameters that we denote collectively by ${\bf s}$. The matrix $J(z;{\bf s})$ must satisfy suitable assumptions (see \cite{BertolaCorrection} for details) but the most important for the description here is the ``local monodromy free'' condition: let $v$ be a ``vertex'' of the graph, namely, a point of intersection of the smooth arcs of $\Sigma$. Let $e_1,\dots e_n$ be the sub-arcs of $\Sigma$ entering a small disk $\mathbb D_v$ centered at $v$ and enumerated counterclockwise from an arbitrarily chosen one.  We denote by
\be
J_\ell(v;{\bf s}) = \lim_{z\to v\atop z\in e_\ell} J^{\pm 1}(z;{\bf s}),
\ee
where the power is $+1$ if the edge $e_\ell$ is oriented away from $v$ and $-1$ viceversa.  Then the matrices must satisfy  
\be
\label{nomono} 
J_1(v;{\bf s})\cdots J_n(v;{\bf s})
 = \Id
\ee
for all the vertices $v$ of $\Sigma$, identically with respect to the deformation parameters ${\bf s}$.
Suppose now that there exists (generically with respect to ${\bf s}$) the solution of the Riemann--Hilbert problem\footnote{To simplify the mental picture, the reader may assume here  that $\Sigma$ is compact: if some rays extend to infinity, the assumption is that $J(z)$ tends to the identity matrix faster than any power of $z^{-1}$ as $z\to \infty$, $z\in \Sigma$, so that the RHP can be posed consistently. Details are in \cite{BertolaCorrection}.}
\be
\Gamma_+(z;{\bf s}) = \Gamma_-(z;{\bf s}) J(z;{\bf s}) ,\ \ \ z\in \Sigma, \ \ \ \ \Gamma(\infty;{\bf s}) \equiv C_0.
\ee
The normalization condition at $z=\infty$ is usually taken to be the identity, but it will be convenient to consider a more general one.
Then we have the definition 
\begin{defn}
\label{defMalg}
The Malgrange form is defined by the formula 
\be
\Theta_{M}:= \int_\Sigma \tr \bigg( \Gamma_-^{-1}(z;{\bf s}) \Gamma_-'(z;{\bf s})\Xi(z;{\bf s}  )\bigg)\frac {\d z}{2i\pi}
\ee
where $\Xi(z;{\bf s}):=  \delta J(z;{\bf s})  J^{-1}(z;{\bf s})$ is the Maurer--Cartan form,  the prime denotes $\d/\d z$ and $\delta$ is the total differential in the deformation parameters ${\bf s}$. 
\end{defn}
We observe that the Malgrange form $\Theta_M$ is independent of the normalization at $z=\infty$, which corresponds to a left multiplication of $\Gamma$ by a $z$--independent matrix. 
Then one has 
\begin{thm}[Thm. 2.1 in \cite{BertolaCorrection}]
\label{thmshame}
The exterior derivative of the Malgrange form $\Theta_M$ is 
\be
\label{deltaTheta}
\delta \Theta_M =  -\frac 1 2 \int_{\Sigma}\frac{\d z}{2i\pi} \tr \le( \Xi'(z) \wedge \Xi(z) \ri)
- \frac 1{4i\pi} \sum_{v\in \mathbf V(\Sigma)}\sum_{\ell=1}^{n_v} \tr\bigg(K_{\ell}^{-1}(v) \delta K_{\ell}(v) \wedge J_\ell^{-1}(v)\delta J_\ell(v) \bigg)
\ee
where $K_\ell(v)= J_1(v)\cdots J_\ell(v)$ and the matrices $J_\ell(v)$ are defined prior to \eqref{nomono}.\footnote{In loc. cit. the form is presented in a different, but equivalent, way.}
\end{thm}

\begin{figure}
\begin{center}
\begin{tikzpicture}

\draw [directed, fill=black!5!white] (-1.5,0) circle[radius =0.5];
\node at (-1.5,0.2) {$\mathbb D_\beta$};
\draw[fill] (-1.5,0) circle [radius=0.02 ];
\node [below] at (-1.5,0) {$0$};
\node at (-1.9,-0.6) {$ z^{-L}$};
\draw [directed] (0:1) to node[above, pos=0.7] {${\rm e}^{2i\pi L}$} (-1,0);
\draw[fill] (-1,0) circle [radius=0.025];
\node [below right] at (-1,0) {$\beta$};
\foreach \angle\pos\label [count=\n] in {-22.5/above/{$\varpi_1$}, 22.5/above/{$\varpi_2$}
, 67.5/below/{$\varpi_3$}  
, 112.5/below/{$ $}  
, 157.5/below/{$ $}  
, -157.5/below/{$ $}  
, -112.5/below/{$ $}  
, -67.5/below/{$\varpi_{_{2K+2}} $}  
}
{
\draw [directed] (1,0) to node[\pos, sloped, pos=0.8] {\label} ++($(\angle:3.5 + 0.017*abs \angle )$);
}
\node [right] at (1.2,0.06) {$z=1$};
\node at ($(1,0)+(0:3.5)$) {$\mathcal S_1$};
\node at ($(1,0)+(45:3.5)$) {$\mathcal S_2$};
\node at ($(1,0)+(-45:3.5)$) {$\mathcal S_{_{2K+2}}$};
\draw [->,dotted]  (-1.5,-5) to (-1.5,5);
\draw [->,dotted]  (-5,0) to (5,0);
\end{tikzpicture}
\end{center}
\caption{An example of Stokes' graph $\Sigma$ used in Theorem \ref{thmMalgrange}.}
\label{figSigma}

\end{figure}
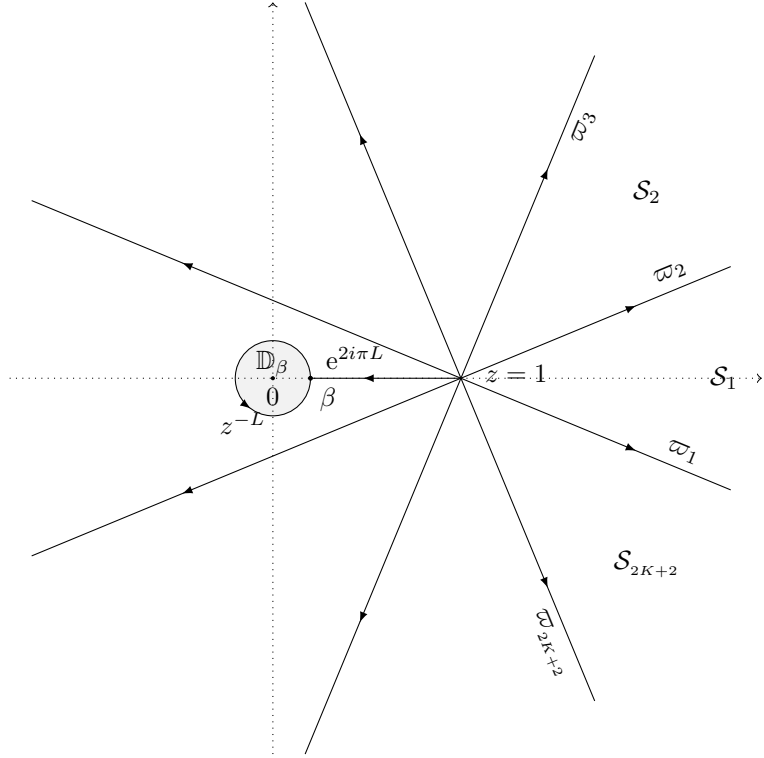

We now come to the main statement of the section.
\begin{thm}
\label{thmMalgrange}
The following  two--form is a (complex) symplectic structure on $\mathfrak S_K$:
\be
\label{OmegaStokes}
{\mathcal W_K} := \frac 1 2 \sum_{\ell=1}^{2K+3} \tr\bigg(
 K_\ell^{-1} \d K_\ell
 \wedge 
 S_\ell^{-1} \d S_\ell  \bigg),\ \ \ K_\ell:= S_1\cdots S_\ell, \ \ S_{2K+3}:= {\rm e}^{2i\pi L}.
\ee
Its pull-back by the (extended) monodromy map coincides with the Lie--Poisson structure \eqref{kricphong} times $-2i\pi$. 
\end{thm}
Before discussing the proof, we point out that this form is written in a different way from \cite{BoalchDuke} (Thm 5, formula (7)) and rather reflects the general theory of ``canonical form associated to a graph'' developed in \cite{Bertola:2019ws}. The two expressions (a posteriori) can be verified to give the same two--form when restricted to the constraint \eqref{nomonodromy}. 
In principle, in our explicit computation later on for the $SL_2$ case, this theorem is verified {\it ex post facto}. 

\noindent {\bf Proof.}
We show that the symplectic form $\mathcal W$  \eqref{kricphong} coincides with the  pull-back by the monodromy map of the  form \eqref{OmegaStokes} and hence showing that the latter is also symplectic (or, to put it more plainly, we write \eqref{kricphong}  in the coordinates provided by the Stokes' matrices).
The proof here is completely different from \cite{BoalchDuke}; rather than computing the two--form $\mathcal W$ in the coordinates of the Stokes' matrices, we directly compute the symplectic potential \eqref{potKKS}.

 Let $\Sigma$ be graph indicated in Fig. \ref{figSigma}: the vertex of the star is at $z=1$ and the small circle is centered at the origin $z=0$.  The Stokes' rays are the lines $\varpi_1,\dots \varpi_{2K+2}$ issuing from $z=1$ and extending to infinity along the Stokes' directions. In the Fig. \ref{figSigma} we have drawn them for the case $K=3$ under the assumption that the real parts $\Re (i t_j)$ are ordered increasingly, so that the Stokes' rays $\varpi_\ell$ have asymptotic directions $\arg z = \frac {i\pi}{2(K+1)} + \frac {i\pi}{K+1} (\ell-1)$ and the Stokes' matrix $S_1$ is then upper triangular.

We now define a piecewise analytic function  $\Gamma$ in each of the connected components of $\C \setminus \Sigma$; in the sector $\mathcal S_1$  $\Gamma$ is given by
\be
\Gamma(z)= \Gamma_1(z) := \Psi_1(z) {\rm e}^{-T(z) +T(1)}  z^{L},
\ee
where the determination of $z^{L}$ is the principal one.
 In the other unbounded components (including the one that contains the disk $\mathbb D_r$) the matrix $\Gamma$ is defined by multiplying $\Gamma_1(z)$ by the jump matrices 
\be
\label{Jell}
J_\ell(z) := {\rm e}^{T(z)-T(1)} z^{-L} S_\ell z^{L} {\rm e}^{-T(z)+T(1)}, \ \ z\in \varpi_\ell.
\ee
The triangularity of $S_\ell$ is such that $J_\ell(z) =\Id +  \mathcal O(z^{-\infty})$ as $|z|\to\infty, \ z\in \varpi_\ell$. 
Within the disk $\mathbb D_r$ we define 
\be
\Gamma(z) = \Gamma_0(z) := \Gamma_{j_0}(z) z^{-L}   = \Psi_{j_0}(z) {\rm e}^{T(1)-T(z)}, 
\ee
where $j_0$ is the index of the sector containing $\mathbb D_\beta$. Note that $\Gamma_0$ is locally analytic near $z=0$.

In the sector containing the disk $\mathbb D_\beta$ the matrix $\Gamma$ does not have a jump on the ray $(-\infty, -\beta]$ because of the monodromy relation \eqref{nomonodromy} and combined with the monodromy of the factor $z^L$.
There is, however the jump $\Lambda = {\rm e}^{2i\pi L}$ on the segment $[\beta,1]$. 
A straightforward exercise shows that the piecewise analytic matrix function $\Gamma$ satisfies a RHP on  the graph $\Sigma$ shown in Fig. \ref{figSigma}:
\be
\Gamma_+(z)& = \Gamma_-(z) J (z),\ \ \ z\in\ \Sigma 
,\qquad \qquad 
\Gamma(z)\simeq {\rm e}^{T(1)}\wh Y(z),\ \ |z|\to\infty,
\label{RHP1}
\ee
where $\simeq$ denotes the asymptotic equivalence in the Poincar\'e\ sense,  $\wh Y(z)$ is the formal series as in Lemma \ref{lemmaeig} and  the jump matrix $J(z)$ is given by
\be
\label{jump}
J(z) = \le\{
\begin{array}{cl}
J_\ell(z) & z\in \varpi_\ell \  \ \text { (see \eqref{Jell})}\\
z^{-L}  & z\in \partial \mathbb D_\beta.
\end{array}
\ri. .
\ee
The jump matrix on $\pa \mathbb D_\beta$ is the function $z^{-L}$ and the determination is (recall that $\beta\in \mathbb R_+$) with $\arg z\in [0,2\pi)$, which is not the same used earlier but we do not want to overload the notation  by using a different symbol for the power.

Using Lemma \ref{lemmaeig} we can write the symplectic potential \eqref{potKKS} as the formal residue 
\be
\theta =& \res_{z=\infty}\ \tr\bigg(A(z) \delta G(z) G^{-1}(z)\bigg)\d z= ``{\res_{z=\infty}}"\ \tr\bigg(A(z) \delta \wh Y(z) \wh Y^{-1}(z)\bigg)\d z
=\cr
=&\res_{z=\infty}\ \tr\bigg(  A(z)\delta \Gamma(z) \Gamma^{-1}(z)   - \Gamma^{-1}(z) A(z)\Gamma(z) \delta T(1)  \bigg)\d z .
\label{2345}
\ee
Since the expansion at $\infty$ of $\Gamma$ coincides with that of the eigenvectors up to order $z^{-K-1}$ (included),  the second term in the residue yields  (recall that $\res_{z=\infty}$ extracts the coefficient of $z^{-1}$ with a {\it minus} sign)
\be
-\res_{z=\infty}\ \tr\bigg( \le(T'-\frac L z\ri) \delta T(1)  \bigg)\d z  =- \tr (L \delta T(1)).
\ee
The first term in \eqref{2345} is a formal residue and can be realized as the following limit of an actual integral
\be
 \lim_{r\to\infty} \oint_{|z|=r} \frac {\d z}{2i\pi}\tr\bigg(A(z) \delta \Gamma \Gamma^{-1}\bigg)
\ee
where the contour runs counterclockwise.  Note that the integrand is actually an analytic function defined piecewisely for each sector.
Applying Cauchy's theorem, we can reduce the integration along the support of the jumps of $\Gamma$ and we obtain
\be
\label{27}
\theta = \int_\Sigma \frac{\d z}{2i\pi} \tr \bigg(A(z) \Delta_{\Sigma} (\delta\Gamma \Gamma^{-1})
\bigg) - \tr \big(L\delta T(1)\big)
\ee
where $\Delta_\Sigma$ is the jump operator $\Delta_\Sigma F(z) = F_+(z)-F_-(z)$,  $z\in \Sigma$.
Now observe that 
\be
\Gamma_+ = \Gamma_- J \ \ \ \  \Rightarrow \  \ \ \ \delta \Gamma_+ = \delta\Gamma_- J + \Gamma_- \delta J
\ \ \ \  \Rightarrow\ \ \ \ \delta \Gamma_+ \Gamma_+^{-1} = \delta \Gamma_- \Gamma_-^{-1} + \Gamma_- \delta J J^{-1}\Gamma_-^{-1}.
\ee
and hence we have 
\be
\label{nice}
\Delta_\Sigma (\delta \Gamma \Gamma^{-1}) = \Gamma_- \delta J J^{-1}\Gamma_-^{-1}.
\ee
Plugging \eqref{nice} into \eqref{27} gives
\be
\label{30}
\theta =   \int_\Sigma \frac{\d z}{2i\pi} \tr \bigg(\Gamma_-^{-1} A  \Gamma_- \delta J J^{-1}\bigg) 
- \tr \big(L\delta T(1)\big).
\ee
The above expression suggest a relationship with the Malgrange form $\Theta_M$  in Def. \ref{defMalg} which we now investigate.
Using the definition $\Gamma(z) = \Psi(z) {\rm e}^{T(1)-T(z)} z^{L}$ (piecewise sectorially), we find that 
\be
A(z) \Gamma(z) = \Psi'(z) {\rm e}^{T(1)-T(z)} z^{L}= \Gamma'(z) + \Gamma \le(T'(z)- \frac {L} z\ri).
\ee
Thus the expression \eqref{30} is recast into:
\be
\label{34}
\theta = \int_\Sigma \frac{\d z}{2i\pi} \tr \bigg(\Gamma_-^{-1} \Gamma_-' \delta J J^{-1}\bigg)
+
\int_\Sigma \frac{\d z}{2i\pi} \tr \bigg(\le(T'(z)- \frac {L} z \ri) \delta J J^{-1}\bigg) - \tr \big(L\delta T(1)\big)
\ee
The integrand in the second integral is zero on each of the Stokes' rays $\varpi_\ell$ because the matrices $\delta J_\ell J_\ell^{-1}$ are strictly triangular (upper or lower), with zeros on the diagonal and $L, T'$ are diagonal, so that the product is diagonal-free. Thus the second integral reduces to 
\be
\int_\Sigma &\frac{\d z}{2i\pi} \tr \bigg(\le(T'(z) -  \frac {L} z \ri) \delta J J^{-1}\bigg)
=\cr
=&\int_\beta^{\beta{\rm e}^{2i\pi}}  \frac{\d z}{2i\pi} \tr \bigg(\le(T'(z) - \frac {L} z \ri)\Big(
-\delta T(1)-\delta L \ln z\Big) \bigg) 
+ \int_1^\beta \bigg(\le(T'(z)- \frac {L} z \ri) \delta L\bigg)\d z
=\nn \\
=& - \sum_{j=1}^{K+1}\frac{\tr (T_j\delta L)}{2i\pi}\le(\frac{\ln z}{j} - \frac {1}{j^2}\ri) z^{j} \bigg|_{\beta}^{\beta {\rm e}^{2i\pi}}
 +  \frac{\delta \tr (L^2)}{4i\pi} \frac{(\ln z)^2}2 \bigg|_{\beta}^{\beta{\rm e}^{2i\pi}} 
 + \tr \Big(L \delta T(1)\Big)\cr
& +\tr \bigg(\big(T(\beta)-T(1)\big)\delta L - \delta \le(\frac {L^2}2\ri) \ln \beta\bigg)
 =\nn\\
 =&- \tr \big(T(1) \delta L\big)  -\frac{i\pi}2  {\delta \tr \big(L^2\big)}.
\ee
Thus we have shown that 
\be
\label{thetaTheta}
\theta =   \Theta_M  -  \tr \big(T(1) \delta L\big)  -2i\pi  {\delta \tr \big(L^2\big)}  -\tr \big(L\delta T(1)\big)=
\Theta_M  -  \delta \tr \le(T(1)  L+\frac {i\pi}2L^2\ri) .  
\ee
This means that the Kirillov-Kostant form $\theta$ coincides with the Malgrange form up to an exact differential.
We now compute the exterior derivative of $\theta$ using Theorem \ref{thmshame}. It is clear that the last term in \eqref{thetaTheta} does not contribute to the exterior differentiation because it is an exact form.  The integral in \eqref{deltaTheta}  has no contribution because 
\begin{enumerate}
\item [-] on the rays $\varpi_\ell$ the integrand is traceless (given the triangularity of the jump matrices \eqref{Jell});
\item[-] on the segment issuing from $z=1$ and directed to the disk, the matrix $\Xi$ is constant in $z$;
\item[-] on the boundary of the disk $\Xi'(z) \wedge \Xi(z) =  \frac {\ln z} z \delta L \wedge \delta L =0$ since $L$ is diagonal.
\end{enumerate}
Thus we are left only with the contributions from the two vertices of the graph in Fig. \ref{figSigma}, which are $v_0 = \beta$ and $v  =1$.
At $v_0$ we have three incident edges and the matrices $J_1,J_2,J_3$ are $J_1 = {\rm e}^{2\pi L}$, $J_2= \beta^{L}$, $J_3 = \beta^{-L} {\rm e}^{-2i\pi L}$. Since they commute, it is easy to see that there is no contribution (each term contains $\delta L \wedge \delta L$, which vanishes identically  since $L$ is diagonal).

Thus the only contribution comes from $v=1$; here the jumps are:
\be
J_\ell(v) =  S_\ell ,\ \ \ \ell=1,\dots, 2K+2
\ee 
and $J_{2K+3} = {\rm e}^{-2i\pi L}$.
Then the Theorem \ref{thmshame} gives precisely \eqref{OmegaStokes} divided by $-2i\pi$. Thus we conclude that ${\mathcal W_K}$ in \eqref{OmegaStokes} is a symplectic form. 
\QED
\begin{rmk}
To be explicit, the coordinates on the quotient of the extended orbit \eqref{extOrbit} are as follows; one writes 
\be
G = G_0\exp\le(
 \frac {H_1}{z} + \frac {H_2} {z^2} + \dots + \frac {H_K} {z^{K}} + \mathcal O(z^{-K-1})
\ri)
\ee
where $H_1,\dots, H_K$ can be chosen {\it diagonal free} (i.e. with zeros on the diagonal), using the gauge freedom \eqref{gauge}. Then the $Kn(n-1)$ entries of $H_1,\dots, H_K$ are the coordinates.
\end{rmk}

\paragraph{The star-graph for the Stokes' phenomenon.}
{
Given the formula \eqref{OmegaStokes} we surmise that the form $\mathcal W_K$ can be represented as $\mathcal W_K = 2 \Omega(\Sigma^\star)$  where $\Sigma^\star$ (the ``star-graph'') is simply the collection of $2K+3$ rays, each carrying the matrices $J_1:= S_1, \dots, J_{2K+2}:= S_{2K+2}, J_{2K+3}:=  \Lambda= {\rm e}^{2i\pi L}$ as jumps. We can actually merge the last two rays and corresponding jump matrices to obtain a simpler star-graph $\Sigma^{(K)}$ indicated by the way of example in Fig. \ref{fig:stokes graph formal mon} for $K=2$. This is not quite one of the generally allowed moves listed in \cite{Bertola:2019ws} but we now verify directly  that it leaves the form invariant. Let thus $\wt J_\ell = J_\ell, \ \ \ell= 1,\dots, 2K+1$ and $\wt J_{2K+2}:= J_{2K+2}J_{2K+3} = S_{2K+2} \Lambda$. Recall that $S_{2K+2}\in N_-$ and $\Lambda$ is diagonal. 
Note that $K_{\ell} = \wt K_{\ell}$ up to $\ell= 2K+1$, while $\wt K_{2K+2} = K_{2K+2} \Lambda=\Id$. 
Then the difference between the two forms is 
\be
\Omega(\Sigma^\star) - \Omega(\Sigma^{(K)}) = \tr \big(K_{2K+2}^{-1} \d K_{2K+2}\wedge S_{2K+2}^{-1} \d S_{2K+2}\big).\label{248}
\ee
Since $K_{2K+3} = K_{2K+2}\Lambda = \Id$ we must have that $K_{2K+2} = \Lambda^{-1}$, namely, it is diagonal. 
But $S_{2K+2}$ is unipotent triangular and hence $S_{2K+2}^{-1} \d S_{2K+2}$ is strictly lower triangular, so that the matrix in \eqref{248} is diagonal--free and the trace gives zero.
Thus, in conclusion, we only  need to analyze the two--form associated to the graphs of the form $\Sigma^{(K)}$ depicted in Fig. \ref{fig:stokes graph formal mon}. We do so for the rank-two case ($n=2$) in the next section. 
}

\section{Stokes manifolds for $n=2$}
\label{sec3}
Our goal now is twofold:
\begin{enumerate}
\item provide explicit parametrization in terms of patches of free coordinates for the  complex manifold $\mathfrak S_K$ \eqref{StokesMfld};
\item show that the coordinates introduced above are log--canonical for two--form \eqref{OmegaStokes}.
\end{enumerate}
We recall here the terminology; a coordinate system $(x_1,\dots, x_{2n})$ on a symplectic manifold $(\mathcal M,\omega)$ is called {\it log-canonical} if the symplectic form   is expressed as follows in the coordinate system
\be
\omega ({\bf x}) = \sum_{i<j} \omega_{ij} \frac {\d x_i}{x_i} \wedge \frac  {\d x_j}{x_j}
\ee
with $\omega_{ij}$ constants. If $P_{ij}$ denotes the inverse transposed of the matrix $\omega_{ij}$ then the Poisson brackets read 
\be
\{x_i,x_j\} = P_{ij} x_ix_j  \ \ \ \text {(no summation)},
\ee
namely the logarithms of the coordinates have constant Poisson brackets amongst themselves (whence the terminology).
At this point the problem of finding Darboux coordinates reduces to a simple problem of linear transformation in the logarithmic coordinates to find the canonical symplectic matrix for the Poisson brackets.

We are going to carry out the two steps above in the case of $SL_2$, which corresponds to the historically first case ever studied in \cite{FlaschkaNewell}. The higher case can be handled in a similar way but we defer the computation to a later paper since it would  unnecessarily obfuscate the computation behind a plethora of indices.

The Stokes' manifold \eqref{StokesMfld} specializes  for any $K \geq 1$ to the following
\begin{equation}\label{eq:stokesmanifold}
\mathfrak{S}_{K}=\left\lbrace 
\begin{pmatrix}
1&s_{1}\\
0&1
\end{pmatrix} 
\begin{pmatrix}
1&0\\
s_{2}&1
\end{pmatrix} \dots
\begin{pmatrix}
1&s_{2K+1}\\
0&1
\end{pmatrix}
\begin{pmatrix}
1&0\\
s_{2K+2}&1
\end{pmatrix} \l^{\s_3} = I_{2} \; \text{with} \; s_i \in \mathbb{C} ,\ \ \l\in \C^\times
\right\rbrace.
\end{equation}
We will denote by $  S_{2l-1} $ the upper triangular matrices and by $  S_{2l} $ the lower triangular matrices appearing in the equation above for $ l=1,\dots,K+1$.   
\begin{rmk}
The matrix equation in \eqref{eq:stokesmanifold} is equivalent to three algebraically independent scalar equations for the Stokes parameters $ s_{j} $ and the formal monodromy exponent $ \alpha $ so that $ \dim \left(\mathfrak{S}_{K} \right)=2(K+1)+1-3=2K$, as it follows from \eqref{eq: general n dim sigma K}  for $n=2$. 
\end{rmk}

\subsection{Computation of ${\mathcal W_K}$}
We consider on $\mathfrak{S}_{K}  $ the 2-form \eqref{OmegaStokes}. 
Following  \cite{Bertola:2019ws} we introduce some basic definitions and properties  of the  $2$-form associated to a graph embedded in a surface, and we will see that the Stokes $2$-form can be conveniently interpreted within that formalism. This is indeed the key in order to compute it explicitly and find the log-canonical coordinates. 
\paragraph{Graph theory}
We  briefly  recall the definition of the standard 2-form associated to an oriented graph on a surface (see Section 2 of \cite{Bertola:2019ws} for more details).
Let $ \Sigma $ be an oriented graph on a surface, we denote with $ \mathbb{V}(\Sigma) $ the set of its vertices, $ \mathbb{E}(\Sigma) $ the set of its edges and $ \mathbb{F}(\Sigma) $ the set of its faces. A ``jump matrix'' $J$ is a map from $\mathbb E(\Sigma)$ to $SL_n$ with the properties that:
\begin{enumerate}
	\item for any edge $ e \in \mathbb{E}(\Sigma) $ we have 
	\begin{equation}
	J(-e)=J(e)^{-1}
	\end{equation}
	with $-e$ denoting the same edge $e$ with opposite orientation;
	\item for any vertex $ v \in \mathbb{V}(\Sigma) $ of valence $ n_{v} $ we have that the ordered counterclockwise product of the matrices associated to each edge oriented away from $ v $ is the identity. Namely:
	\begin{equation}\label{eq:mongraph}
	J(e_{1})\dots J(e_{n_{v}})=I_{n},
	\end{equation}
	where we ordered the edges $ e_{1},\dots , e_{n_{v}} $ incident at $v$ then counting them counterclockwise.
\end{enumerate}
To the pair $ (\Sigma, J) $, we can then associate the standard 2-form $ \Omega(\Sigma) $ defined hereafter.
\begin{defn}\label{def:standard2form}
	The standard 2-form $ \Omega(\Sigma) $ associated to the graph $ \Sigma $ is defined as follows (we omit explicit reference to the dependence on $J$ from the notation) 
	\begin{equation}\label{eq:2formstandard}
	\Omega(\Sigma):=
	\sum_{v \in \mathbb{V}(\Sigma)} \sum_{ k=1}^{n_{v}-1} \Tr \left( 
	 \left( K^{(v)}_{\left[ 1:k\right] }\right) ^{-1}\d K^{(v)}_{\left[ 1:k\right]} \wedge 
	 \left( J^{(v)}_{k}\right) ^{-1}\d J^{(v)}_{k} \right) .
	\end{equation}
	where in this formula  for any vertex $ v \in \mathbb{V}(\Sigma) $ we have taken the incident edges $e_1,\dots,e_{n_v}$ oriented away from $v$ and enumerated in counterclockwise order, starting from any of them. Here $ K^{(v)}_{\left[ 1:k\right]}=J_{1}\dots J_{k}$ with $J_i=J(e_i)$ for $i=1,\dots,n_v.$ Thanks to the property \eqref{eq:mongraph}, this 2-form is well defined, namely, independent of the choice of first edge in the cyclic order at each vertex.
\end{defn}
\begin{figure}[ht]
	\centering
	\includegraphics[width=0.28\textwidth]{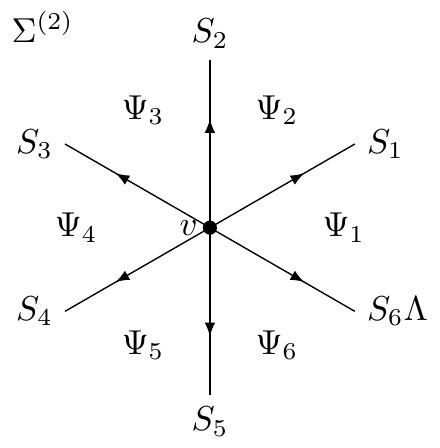}
\caption{The Stokes graph $ \Sigma^{(2)}$. 
}
	\label{fig:stokes graph formal mon}
\end{figure}
The form $\Omega(\Sigma)$ in Def. \ref{def:standard2form} is shown to be  invariant under certain transformations $(\Sigma, J)\mapsto (\Sigma',J')$ (called {\it moves}, see Section 2 of \cite{Bertola:2019ws}); these moves consist in the self--describing titles of 
\begin{enumerate}
\item edge contractions;
\item merging edges;
\item attaching edges to vertices (and the converse)
\end{enumerate}

In order to compute ${\mathcal W_K}$ we proceed as follows:  
it follows directly from the definition \eqref{eq:2formstandard} that the symplectic form ${\mathcal W_K}$ defined by \eqref{OmegaStokes} can be viewed (up to a factor) as the  2-form associated to the simple Stokes graph like the one in  Figure \ref{fig:stokes graph formal mon} for the exemplifying case $K=2$, with the indicated jump matrices. Namely,
\begin{equation}
\label{eq:relationstokes form and graph}
2{\mathcal W_K}= \Omega(\Sigma^{(K)}).
\end{equation}

The idea is to realize the simple graph $\Sigma^{(K)}$ as the complete contraction of all the (finite length) edges of another graph with explicit, simple jump matrices  that depend on {\it free} parameters (contrary to the Stokes' parameter that are subject to algebraic relations).

Consider the  graph $\Sigma_{0}^{(K)}$, exemplified  in Figure \ref{fig:mod graph form monodromy} for $K=2$: then   it is apparent that $\Sigma^{(K)}$  is the total contraction of  $\Sigma_{0}^{(K)}$.
The jump matrices for this graph are described in the following paragraph. 
The key fact is that the computation of the symplectic form associated to $\Sigma_0^{(K)}$ is then a straightforward exercise.  

Since the graphs $\Sigma_0^{(K)}$ and $\Sigma^{(K)}$ are related by the ``moves'' hinted at before and described in \cite{Bertola:2019ws}, the corresponding associated forms coincide:
 $\Omega\left( \Sigma_{0}^{(K)}\right) =\Omega\left(\Sigma^{(K)} \right) $.
  Then, by using the definition of the 2-form associated to a graph, we will compute explicitly the Stokes form, showing directly that it is indeed symplectic.
 
\begin{figure}
	\centering
\includegraphics[width=0.4\textwidth]{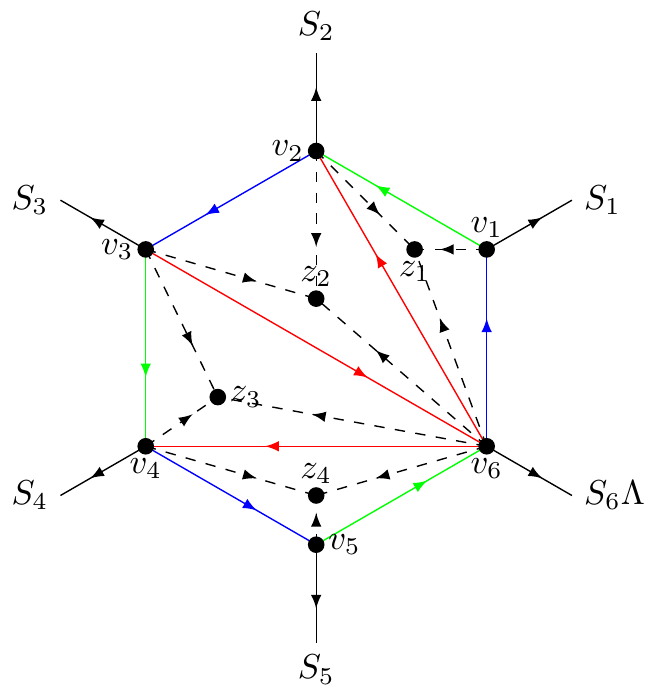}
	\caption{The modified graph $ \Sigma^{(2)}_{0} $. Here we take the triangulation $T_0$ of the hexagon that connects any of its vertices to $ v_{6}. $}
	\label{fig:mod graph form monodromy}
\end{figure}
\paragraph{The graph $\Sigma_{0}^{(K)}$ and its jump matrices.} The graph $\Sigma_0^{(K)}$ (see Fig. \ref{fig:mod graph form monodromy} for the example with $K=2$) is the graph consisting of $2(K+1)$ infinite rays emanating from the vertices of a regular $2K+2$--gon. The polygon is subdivided into triangles with a common vertex $v_{2K+2}$. We denote by $T_0$ this precise triangulation of the polygon. Inside each triangle we have a vertex $z_j$ and three edges from the three vertices bounding the triangle to the vertex $z_j$.  
We describe the jump matrices for $T_0$ with the understanding that, {\it mutatis mutandis}, the same matrices are defined for an arbitrary triangulation.
To each oriented edge of $ \Sigma_{0}^{(K)} $ we associate a matrix that is constant or depends on complex parameters $ y_{j} \in \mathbb{C}^{*}, j=1,\dots, 2K.$ The orientation is defined as follows: the perimeter of the polygon is oriented  counterclockwise  and as for the vertices $ z_{j} $, each edge is oriented towards the vertex $ z_{j} $. The internal diagonals of the triangulation are oriented in such a way that for every even perimetric vertex there is an even number on incident diagonals and for every odd vertex there is an odd number of incident diagonals. The Stokes rays are kept with the same orientation as in the Stokes graph. The matrices for each edge are defined as follows:
\begin{itemize}
	\item on the perimetric  edges connecting  $ v_{2k}\rightarrow v_{2k+1}$ for $k=1,\dots K$ and $v_{2K+2}\rightarrow v_1 \sim v_{2K+3}$ (the blue edges in Figure \ref{fig:mod graph form monodromy}), we take diagonal matrices of the form
	\begin{equation}\label{eq:blue perim matrix}
	D\left( x_{2k}\right):=\begin{pmatrix}
	 x_{2k}^{-1}&0\\
	0& x_{2k}
	\end{pmatrix} , 
	\end{equation}
	where $  x_{l} $ is the following product of $y_j$'s variables
	\begin{equation}
	\label{eq:x variable definition}
	x_{l} \coloneqq \prod_{1\leq k\leq l} \prod_{d_j \perp v_k}y_j^{(-1)^{k+1}},\;\; l=2,\dots,2K+1, \;\;\;x_{2K+2} \coloneqq y_1\prod_{d_j \perp v_1}y_j^{-1}
	\end{equation}
    \item on the perimetric edges connecting $  v_{2k+1}\rightarrow v_{2k+2}$ (the green edges in Figure \ref{fig:mod graph form monodromy}), we take off-diagonal matrices of the form
	\begin{equation}\label{eq:green perim matrix}
V\left( x_{2k+1}^{-1}\right)
:=\begin{pmatrix}
0&-x_{2k+1}^{-1}\\
x_{2k+1}& 0
\end{pmatrix},
\end{equation}
 and along the edge $v_1 \rightarrow v_2$ we impose the jump matrix $V(y_1^{-1})$;
	\item on the three edges incident to $ z_{j} $ (each of the dashed lines in Figure \ref{fig:mod graph form monodromy}) we associate the constant matrix 
	\begin{equation}
	A:=\begin{pmatrix}
	0&1\\
	-1&-1
	\end{pmatrix},
	\end{equation}
	that has the property $A^3=\Id.$
	\begin{rmk}
	In the $SL_n$ case, the matrix $A$ would be replaced by matrices $A_{1,2,3}$ that depend on $(n-1)(n-2)/2$ additional parameters for each triangle. These matrices are used in Sec. \ref{secUgaglia}.
	\end{rmk}

	\item on each internal diagonal edge $d_j$ for $j=2,\dots,2K$ defining the original triangulation $T$, we associate  off-diagonal matrices of the form $ V (y_j)$ given by
	\begin{equation}
	V (y_j)  :=\begin{pmatrix}
	0&-y_{j}\\
	y_{j}^{-1}&0
	\end{pmatrix} 
	\end{equation} 
	for $j=2,\dots,2K$ (these are the red edges of Figure \ref{fig:mod graph form monodromy}).  In this way each internal diagonal $d_j$ is uniquely associated to the free variable $y_j$, for $j=2,\dots,2K$.
\end{itemize}
\begin{rmk}
In this construction there  one among the boundary edges plays a distinguished role, namely, the one  laying to the left of the first Stokes ray. Indeed, the matrix associated to this edge is of the same type of the matrices associated to the internal diagonal edges of the triangulation $T$ and it depends only  on $y_1$.
\end{rmk}
The Stokes' matrices $S_j$ on the unbounded rays are then uniquely determined in terms of the remaining ones by the condition \eqref{eq:mongraph} at the corresponding vertex $v_j$. In this way each $ S_i $ is expressed in terms of the $y_j$ variables. Of course, for each triangulation, we will obtain different parametrization of the Stokes parameters and the transformation of coordinates will be investigated later.
\paragraph{The initial triangulation}
Consider now the triangulation $T_0$, underlying the graph $\Sigma_0^{(K)}$, where the last vertex $v_{2K+2}$ is connected to each other vertex starting from $v_2$, and with alternated orientation of the internal diagonals (as in Figure \ref{fig:mod graph form monodromy} for the case $K=2$). Then the Stokes matrices are given by 
\begin{equation}\label{eq:parametrization of Stokes matrices}
\begin{aligned}
&S_1 = \left(V(y_1^{-1}) A D(y_1)^{-1}\right) ^{-1}\\
&S_2 = \left( D(x_{2})AV(y_{2})^{-1}AV(y_{1}^{-1})^{-1}\right)^{-1},\\
& S_{2k}= \left( D(x_{2k})AV(y_{2k})^{-1}AV(x_{2k-1}^{-1})^{-1}\right)^{-1},\;\; k=2,\dots, K\\
& S_{2k+1}= \left( V(x_{2k+1}^{-1})AV(y_{2k+1}))A D(x_{2k})^{-1}\right) ^{-1},\;\; k=1,\dots, K-1\\
&S_{2K+1} =\left(  V(x_{2K+1}^{-1})A D(x_{2K})^{-1}\right) ^{-1}\\
&S_{2K+2}\Lambda = \left( D(y_1)\prod_{j=2}^{2K}\left(A V(y_j)^{(-1)^j} \right) A V(x_{2K+1}^{-1})^{-1}\right) ^{-1}
\end{aligned}
\end{equation}
The choice of the triangulation $T$  defines also the variables $x_l$. According to the general rule \eqref{eq:x variable definition} with the triangulation $T_0$ fixed here, this definition reduces to
\begin{equation}
\label{eq:variables x }
x_{l} \coloneqq \prod_{j=1}^{l}y_{j}^{(-1)^{j+1}} \;\; l=2,\dots, 2K,\;\;\; x_{2K+1}=x_{2K},\;\; x_{2K+2}\coloneqq y_1.
\end{equation}
These considerations are summarized in the following lemma.
\begin{prop}
\label{prop:1}
The Stokes parameters are written in terms of the $y_j$ variables, w.r.t. the fixed triangulation $T_0$ described above, as follows
\be
\label{eq:param stokes param y}
&s_1 = -y_1 ^{-2}\cr 
&s_{2k} = (1+y_{2k}^2) \prod_{1\leq j\leq 2k} y_{j}^{(-1)^{j+1}2},\;\; k=1,\dots , K \cr
& s_{2k+1} = -(1+y_{2k+1}^2) \prod_{1\leq j\leq 2k+1} y_{j}^{(-1)^j 2}, \;\; k=1,\dots , K-1\cr
&s_{2K+1} = - \prod_{1\leq j\leq 2K} y_{j}^{(-1)^j 2} ,\cr
&s_{2K+2} = y_1^{2}\left( 1+y_2^2\left( \dots \left(1+y_{2K}^2\right) \dots \right) \right) \prod_{j=1}^{K} y_{2j}^{-4}, \cr
&\lambda = \prod_{j=1}^{K}y_{2j}^2 .
\ee
\end{prop}
\begin{proof}
Just computing explicitly the parametrizations given from equations \eqref{eq:parametrization of Stokes matrices} and using the definition of the variables $x_{2k}, x_{2k+1}$ given in \eqref{eq:variables x }. 
\end{proof}
With this parametrization of the Stokes matrices we can then proceed to the computation of the Stokes form. 
\begin{prop}
The  2-form associated to the graph $\Sigma^{(K)}_0$ coincide with
\begin{equation}\label{eq:sympl 2 form in y variables}
\Omega  \left( \Sigma^{(K)}_0\right)  = 
{+}8 \sum_{\substack{j=1\\ l \geq j}}^{K} d\log y_{2j-1} \wedge d\log y_{2l}.
\end{equation}
In particular it is symplectic.
\end{prop}
\begin{proof}
The fact that the form is symplectic follows from  Theorem \ref{thmMalgrange} and the fact that the contraction of  $\Sigma_0^{(K)}$ coincides with the graph $\Sigma_K$  (see Fig. \ref{fig:stokes graph formal mon}); however the explicit expression \eqref{eq:sympl 2 form in y variables} is manifestly a nondegenerate form and so it could be used directly as a proof. 
By using the definition of the 2-form \eqref{eq:2formstandard}, we have to compute the contributions coming from each vertex $v_j, j=1, \dots 2K+2$ in the graph $\Sigma^{(K)}_0$. The vertices $z_j, j=1,\dots, 2K$ do not give any contribution since all their incident edges carry constant matrices. \\
We  start with the vertex $v_1$. Since the valence of $v_1$ is 4 and $A$ is a constant matrix, there is only one contribution to take into account from $v_1$, and it is
\begin{equation}
\Tr \left( 
\underbrace{ \left( V(y_1^{-1})AD(y_1)^{-1}\right)^{-1} d\left( V(y_1^{-1})AD(y_1)^{-1}\right)    }_{=S_1 d(S_1^{-1})}
\wedge
\underbrace{\left( D(y_1)d D(y_1)^{-1} \right) }_{=-d\log y_1 \sigma_3} \right)  =0
\end{equation}
that turns out to be also zero, thanks to the form of the Stokes matrices given in  \eqref{eq:parametrization of Stokes matrices}. Thus the total contribution of the vertex $v_1$ is actually zero.\\ 
Since the vertex $v_{2K+1}$ is in the same configuration of $v_1$, but replacing $D(y_1)$ by $D(x_{2K})$, by the same reasoning we can conclude that its contribution is also zero. \\
Now we compute the contributions of the vertices $v_{2k}$ for $k=1,\dots , K$. For each of them there is only one nonzero contribution and it is coming from the term
\begin{equation}
\begin{aligned}
&\Tr \left( 
\underbrace{\left(\left(D(x_{2k})AV(y_{2k})^{-1} \right)^{-1} d\left(D(x_{2k})AV(y_{2k})^{-1} \right)\right) }_{-d\log(x_{2k-1}) + E_{21}f(\vec{y})d\vec{y}}
\wedge
\underbrace{\left( V(y_{2k}) d(V(y_{2k})^{-1}) \right) }_{= - d\log y_{2k} \sigma_3} 
   \right) =\\ \\
&= {} 2 d\log x_{2k-1}\wedge d\log y_{2k}=\\
&={}2d\log \left(\prod_{j=1}^{2k-1}y_{j}^{(-1)^{j+1}} \right) \wedge d\log y_{2k} =\\
&={}
2d\log y_1 \wedge d\log y_{2k} 
{}
2 \sum_{l=2}^{k}d\log y_{2l-1} \wedge d \log y_{2k} 
{-}
2 \sum_{l=1}^{k-1} d\log y _{2l}\wedge d\log y_{2k}.
\end{aligned}
\end{equation}
Notice that for the case $k=1$ we only have the term $2d\log y_1 \wedge d\log y_{2}$.\\
A similar computation shows that the only nonzero contribution for the vertices $v_{2k+1}$ for $k=1,\dots, K-1$ is given by
\begin{equation}
\begin{aligned}
&\Tr \left(
\underbrace{\left(\left(V(x_{2k+1}^{-1})AV(y_{2k+1}) \right)^{-1} d \left(D(x_{2k+1})JAV(y_{2k+1})   \right)   \right)}_{=d\log x_{2k}\sigma_3 + E_{21}g(\vec{y})d\vec{y}} 
\wedge
\underbrace{\left( V(y_{2k+1})^{-1} d(V(y_{2k+1}))\right)}_{=-d\log y_{2k+1}} 
\right)  =\\ 
&={-}2d\log x_{2k} \wedge d\log y_{2k+1} =\\
&= {-}2 d \log \left(\prod_{j=1}^{2k}y_{j}^{(-1)^{j+1}} \right) \wedge d\log y_{2k+1}=\\
&={-} 
2d\log y_1 \wedge d\log y_{2k+1}
{-} 
2\sum_{j=2}^k d\log y_{2j-1}\wedge d\log y_{2k+1} 
{+}
 2 \sum_{j=1}^k d\log y_{2j} \wedge d\log y_{2k+1}.
\end{aligned}
\end{equation}
It  only  remains to compute the contribution of the vertex $v_{2K+2}$. The internal diagonals carrying  the variables $y_{2k}$ for $k=1,\dots,K$ give the contribution,
\begin{equation}
\begin{aligned}
C_1& \coloneqq \sum_{ k=1}^{K} \Tr\left( \left( V(y_{2k})^{-1}d(V(y_{2k}))\right) \wedge \left( D(y_1)\prod_{j=2}^{2k}A \left( V(y_j)\right) ^{(-1)^j}\right)^{-1} d \left( D(y_1)\prod_{j=2}^{2k}A\left(  V(y_j)\right) ^{(-1)^j}\right) \right) =\\
&= \sum_{ k=1}^{K} \Tr \left( -d\log y_{2k}\sigma_3 \wedge \left( -d\log y_1-\sum_{j=2}^{2k} d\log y_j\right)\sigma_3 \right) =\\
&=-2\sum_{k=1}^{K}d\log y_1 \wedge d\log y_{2k}+2 \sum_{ \substack{k=1\\ j \leq k}}^{K}\left( d\log y_{2k}\wedge d\log y_{2j} +d\log y_{2k}\wedge d\log y_{2j-1}\right) .
\end{aligned}
\end{equation}
The internal diagonals carrying on the variables $y_{2k+1}$ give instead the contribution
\begin{equation}
\begin{aligned}
C_2& \coloneqq \sum_{k=1}^{K-1} \Tr \left(
 \left(  D(y_1)\prod_{j=2}^{2k+1}A\left(  V(y_j)\right) ^{(-1)^j}\right) ^{-1} \d \left(  D(y_1)\prod_{j=2}^{2k+1}A\left(  V(y_j)\right) ^{(-1)^j}\right)
\wedge
\left( V(y_{2k+1}) \d(V(y_{2k+1})^{-1})\right) 
  \right) =\\
&= \sum_{ k=1}^{K-1}\Tr\left(
 \left(-d\log y_1-\sum_{j=2}^{2k+1} d\log y_j\right)\sigma_3 
 \wedge
 ( -d\log y_{2k+1} \sigma_3 )
 \right)=\\
&= 
2\sum_{ k=1}^{K-1}d\log y_1 \wedge d\log y_{2k+1}
{-}
 2\sum_{ \substack{k=2\\ j \leq k}}^{K-1}\left( d\log y_{2k+1}\wedge d\log y_{2j} 
+  d\log y_{2k+1}\wedge d\log y_{2j-1} \right) .
\end{aligned} 
\end{equation}
Finally the last edge on the right of the Stokes ray of $v_{2K+2}$ also give a nonzero contribution, that is 
\begin{equation}
\begin{aligned}
C_3&\coloneqq \Tr \left( 
(S_{2K+2}\Lambda) \d(S_{2K+2}\Lambda)^{-1} 
\wedge
\left( V(x_{2K+1}^{-1})\d(V(x_{2K+1}^{-1})^{-1})\right) 
\right) =\\
&= \Tr \left( 
\left( 2 \sum_{ l=1}^{K} \d\log y_{2l}\right)\sigma_3 
\wedge
\left( -\d\log y_1+\sum_{j=1}^{K}\left( - \d\log y_{2j+1} + \d\log y_{2j}\right)  \right)\sigma_3
\right) =\\
&= 
4\sum_{ l=1}^{K}d\log y_1 \wedge d\log y_{2l}  
{+}
4 \sum_{j=2}^{K}d\log y_{2j-1} \wedge \sum_{ l=1}^{K}d\log y_{2l} 
{-}
 4 \underbrace{\sum_{j=1}^{K}d\log y_{2j} \wedge \sum_{ l=1}^{K}d\log y_{2l}}_{=0}=\\
&=
4\sum_{ l=1}^{K}d\log y_1 \wedge d\log y_{2l}
{+}
4 \sum_{j=2}^{K}d\log y_{2j-1} \wedge \sum_{ l=1}^{K}d\log y_{2l} 
\end{aligned}
\end{equation}
where in the last equality we used the skew-symmetry of the wedge product. Now we can sum up all the nonzero  contributions coming from $v_{l}, l=2,\dots,2K+2$ and we obtain

\be
\nn
\Omega\left( \Sigma^{(K)}_0\right)& 
= 
2\sum_{k=1}^{K}\d\log y_1 \wedge \d\log y_{2k} 
{+}
2 \sum_{\substack{2\leq l\leq k \\k =2} }^{K}\d\log y_{2l-1} \wedge \d \log y_{2k} 
{-} 2 \sum_{\substack{1\leq l\leq k-1\\ k =2}}^{K} \d\log y _{2l}\wedge \d\log y_{2k}
\\
&  \nn
{-}
2\sum_{k=1}^{K-1}d\log y_1 \wedge \d\log y_{2k+1}
{-}
2\sum_{\substack{2\leq j\leq k\\ k=2}}^{K-1 }\d\log y_{2j-1}\wedge \d\log y_{2k+1} 
{+} 
 2 \sum_{\substack{1\leq j \leq k\\ k=1} }^{K-1} d\log y_{2j} \wedge d\log y_{2k+1}
\ee
\be
& {+}
2\sum_{k=1}^{K}\d\log y_1 \wedge \d\log y_{2k}
{-}
2 \sum_{ \substack{k=1\\ j \leq k}}^{K}\left( \d\log y_{2k}\wedge \d\log y_{2j} 
+
\d\log y_{2k}\wedge \d\log y_{2j-1}\right)\\
&{+}
2\sum_{ k=1}^{K-1}\d\log y_1 \wedge \d\log y_{2k+1}
{-} 2\sum_{ \substack{k=2\\ j \leq k}}^{K-1}\left( \d\log y_{2k+1}\wedge \d\log y_{2j} +  \d\log y_{2k+1}\wedge \d\log y_{2j-1} \right)\\
&
{+}
4\sum_{ l=1}^{K}\d\log y_1 \wedge \d\log y_{2l}
{+}
4 \sum_{j=2}^{K}\d\log y_{2j-1} \wedge \sum_{ l=1}^{K}\d\log y_{2l} \\
&= 
{-}8\sum_{ \substack{k=1\\ j \geq k}}^{K}\d\log y_{2k-1}\wedge \d\log y_{2j}
\ee

\end{proof}
By using relation \eqref{eq:relationstokes form and graph}, we can finally conclude that the Stokes 2-form ${\mathcal W_K}$ is written in terms of these $y_j$ variables as
\begin{equation}
{\mathcal W_K} = {}\frac{1}{2} \Omega\left( \Sigma^{(K)}\right) =  {}
\frac{1}{2} \Omega\left( \Sigma^{(K)}_0\right) = 
4 \sum_{ \substack{k=1\\ j \geq k}}^{K} d\log y_{2k-1}\wedge d\log y_{2j},
\end{equation}
and since it has maximal rank, it is a symplectic 2-form. \QED  
The Poisson structure induced by the the symplectic structure in the same variables will be then written as 
\begin{equation}\label{eq:poisson str}
\left\lbrace y_i, y_j\right\rbrace = \mathbf{P}^{ij}_{K} y_i y_j
\end{equation}
where $\mathbf{P}_{K} = \mathbf{\Omega}_{K}^{-t}$ 
 and $\mathbf{\Omega}_{K}$ is the matrix of coefficient of the Stokes 2-form w.r.t. the logarithmic variables $\log y_l$.
\begin{lem}
\label{lemmaPK}
The matrix  $\mathbf{P}_{K} $ is the $2K\times2K$ tridiagonal matrix given by
\begin{equation}
\label{eqPK}
\mathbf{P}_{K} = \frac{1}{4}
\begin{pmatrix}
0 & 1 & 0 &0 & 0 & \dots &0\\
-1&0&1 &0 & 0 & \dots &0 \\
0&-1&0&1&0&\dots&0\\
\vdots & &\ddots & \ddots & \ddots& &\vdots\\
\vdots & &&\ddots & \ddots & \ddots& \vdots\\
0&0&\dots & &- 1&0&1\\
0&0&\dots & & 0 & -1 &0
\end{pmatrix}
\end{equation}
\end{lem}
\subsection{Comparison between $\mathbf{P}_K$ and Poisson structure on $Y$-cluster manifold}
Let focus our attention on the matrix $\mathbf{B}_K \coloneqq 4 \mathbf{P}_K$.
\begin{defn}
Given a quiver $Q$ with labeled vertices $q_i, i=1,\dots, \#\mathbb{V}(Q)$, we call $B$ its adjacency matrix the skew-symmetric, integer-valued square matrix, of dimension $\#\mathbb{V}(Q)$, given by 
\begin{equation}
B_{kl}\coloneqq \#\left\lbrace \text{edges oriented from }\;q_k\; \text{to} \;q_l \right\rbrace - \#\left\lbrace\text{ edges oriented from} \;q_l\; \text{to}\; q_k \right\rbrace
\end{equation}
for $k,l=1,\dots, \#\mathbb{V}(Q)$.
\end{defn}
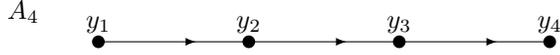
\begin{figure}
	\centering
	\begin{tikzpicture}[xscale=2,yscale=2, >=stealth']
	\draw[>=latex,directed] (-2,0)node[above]{$y_1$} to (-1,0) node[above]{$y_2$};
	\draw[>=latex,directed] (-1,0)  to (0,0)node[above]{$y_3$} ;
	\draw[>=latex,directed] (0,0) to (1,0) node[above]{$y_4$};
	\fill (-2,0)  circle[radius=1.2pt];
	\fill (-1,0)  circle[radius=1.2pt];
	\fill (0,0)  circle[radius=1.2pt];
	\fill (1,0)  circle[radius=1.2pt];
	\node at (-2.5,0.2) {$A_4$};
	\end{tikzpicture}
	\caption{The Dynkin diagram associated to the $4\times4 $ matrix $\mathbf{B}_2 $. This quiver can also be obtained following the construction described in the paragraph below with the triangulation of the hexagon fixed to be $T_0$.}
	\label{fig:dynkdiag}
\end{figure}
Then the matrix $\mathbf{B}_K$ can be identified as the directed adjacency matrix of a Dynkin graph of type $A_{2K}$ with specified orientation. An example for $K=2$ is given in Figure \ref{fig:dynkdiag}. There is a classical way to associate a directed graph to a triangulation of a given polygon (see for instance paragraph $2.1$ of \cite{GSV}). We slightly modify this construction, taking into account the fact that there is an edge along the perimeter of the polygon (the edge at the left of the first Stokes ray) that has a distinguished role in our case. We end up with the following graph $Q(T)$ for a given triangulation $T$ of the polygon:
\begin{itemize}
    \item the vertices of $Q(T)$ are defined one for each of the following edges of $T$: the edge along the perimeter at the left of the first Stokes ray and every internal diagonal edge of the triangulation $T$;
    \item the edges of $Q(T)$ are are build between each pair of vertices that lies on edges of the triangulation $T$ that share one of the endpoints and are immediately adjacent;
    \item the orientation of the edges of $Q(T)$ is defined as follows: an edge connecting the vertices $q_i$ and $q_j$ on the adjacent edges of $T$ $d_i$ and $d_j$ is oriented $q_i \rightarrow q_j$ if the edge $d_i$ immediately precedes $d_j$ counting counterclockwise the edges incident to their common endpoint. Otherwise it is oriented in the opposite way. For the vertex $y_1$ along  the edge on the right of the first Stokes ray (since on this edge we actually used the variable $ y_1^{-1}$ ) we reverse the orientation of all the edges of $Q(T)$ that have $y_1$ as endpoint.
\end{itemize}

With this construction, we obtain that   for the initial triangulation $T_0$ underlying $\Sigma_K^{(0)}$  the quiver $Q(T_0)$ is a Dynkin graph of type $A_{2K}$ with the orientation induced from $T_0$ (but each orientation of the same type of Dynkin graph is mutation equivalent, see Theorem $3.29$ of \cite{GSV}). \\ The matrix $\mathbf{B}_K$ gives  a compatible Poisson structure on the $Y$-cluster manifold which is defined by the ring of functions that are polynomials in all of the seeds obtained by subsequent {\it mutations}, defined below.
\begin{figure}
	\centering
\includegraphics[width=0.4\textwidth]{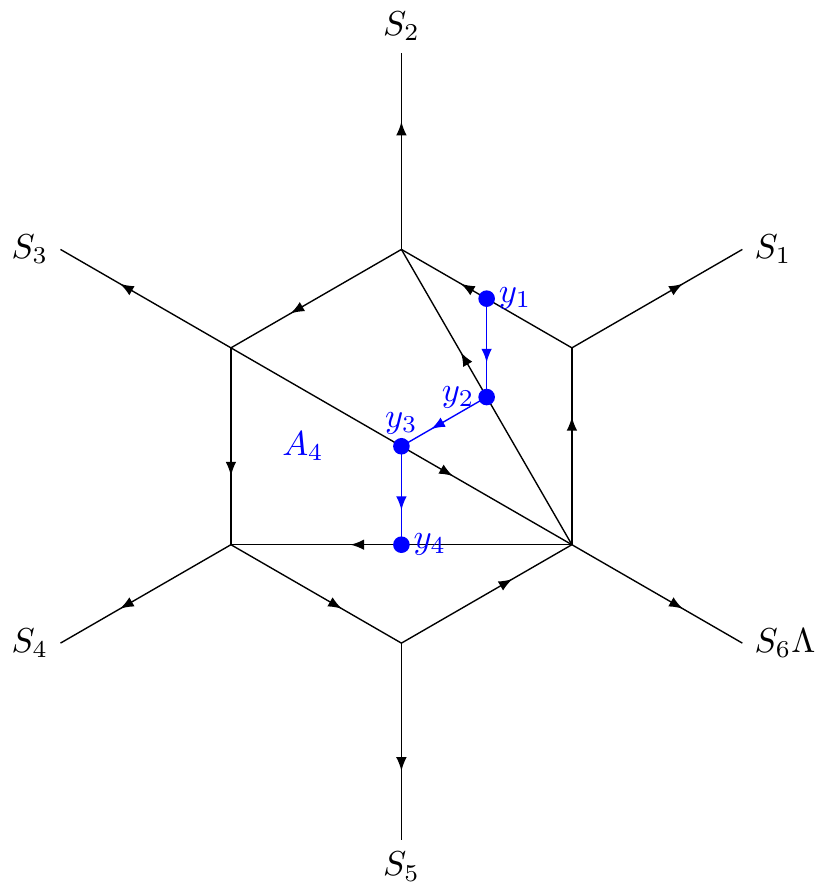}
	\caption{Here the triangulation $T$ of the hexagon and the variables $y_j$ assigned to the relevant edges induce the Dynkin diagram with variables $y_1, y_2, y_3, y_4$ in blue. }
	\label{fig:triang and dynk}
\end{figure}
\begin{defn}
A mutation $\mu_k(Q)$ w.r.t. a vertex $q_k \in \mathbb{V}(Q)$ of the quiver $Q$ is a new quiver defined by
\begin{itemize}
\item the same set of vertices, namely $\mathbb{V}(Q)=\mathbb{V}(\mu_k(Q))$;
\item the set of edges constructed as follows 
\begin{enumerate}
\item for any sequence $q_i \rightarrow q_k \rightarrow q_l$ add an edge $q_i\rightarrow q_l$,
\item reverse any edge having source or end in the vertex $q_k$,
\item remove every $2$-cycle if any.
\end{enumerate}
\end{itemize}
\end{defn}
Equivalently we can define the mutation $\mu_k(Q)$ of $Q$ through its adjacency matrix $\mu_k(B)$ that is given by the following equations
\begin{equation}
\mu_k(B)_{st}= 
\begin{cases}
-B_{st},\\
B_{st} + sign(B_{sk})\left[B_{sk},B_{kt} \right]_+,
\end{cases}\;\;
\begin{aligned}
&\text{for }\; s=k\; \text{or}\; s=t\\
& \text{otherwise}
\end{aligned}
\end{equation}
In our case of study, a set of variables $y_i\in\mathbb{C}^{*}$ one of each vertex $q_i$ is associated  to the quiver, for $ i=1,\dots,2K$. To each mutation $\mu_k(Q)$ of the quiver is then associated a new set of variables $\mu_k(\vec{y})$ following the equations in the definitions that we recall below (see also (1.30) in e.g. \cite{GSV}).
\begin{defn}
\label{defmuty}
A $Y$-mutation for the variables $y_i$ of the couple $(Q,\vec{y})$ is a new set of variables $\left(\mu_k(\vec{y})\right)_{i=1}^{2K}
,$ for $i=1,\dots 2K$ defined as rational functions of the $y_i$ in the following way
\begin{equation}
y_i'\coloneqq \left( \mu_k(\vec{y})\right)_i = \begin{cases}
y_k^{-1},\\
y_i \frac{y_k^{\left[B_{ik} \right]_+ }}{(1+y_k)^{B_{ik}}},
\end{cases}
\begin{aligned}
&\text{for }\;i=k\\
& \text{otherwise}
\end{aligned}
\end{equation}
\end{defn}

Every new pair $\mu_k(\vec{y},Q)=(\vec{y}',Q')$ obtained by an allowed mutation is called a {\it seed}.
In our case, we have that the initial quiver $Q(T)$ is the Dynkin graph of $A_{2K}$-type (for every $n\geq 1$) that is related to the triangulation $T$ of the polygon in $\Sigma^{(K)}_0$. The allowed mutations in this case are  with respect to  all the vertices with variables $y_2,\dots,y_{2K}$ (the ones associated to the internal diagonals of the triangulation $T$ of the polygon).
\begin{defn}
Given a pair $(\vec{y},Q)$ where $Q$ is a quiver with labeled vertices $q_i, i=1,\dots, \#\mathbb{V}(Q)$ and the variables $y_i\in\mathbb{C}^{*}$ are associated to each $q_i$, we call the $Y$-cluster algebra $\mathcal{A}_Y(Q)$ the sub-ring of all polynomials in $y_i$ and all their possible seeds $\mu_k(\vec{y},Q)$ where $\mu_k$ is a mutation w.r.t. the vertex $q_k$ with assigned variable $y_k$.
\end{defn}
\begin{defn}
Given a $Y$-cluster algebra, its correspondent $Y$-cluster manifold is defined as the smooth part of $Spec(\mathcal{A}_Y(Q))$.
\end{defn}
Denoting by $\mathcal{A}_{Y, i\neq 1}(A_{2K})$ the $Y$-cluster algebra described above for our case, then on its correspondent $Y$-cluster manifold  $\mathcal{M}\coloneqq Spec(\mathcal{A}_{Y, i\neq 1}(A_{2K}))$ there is a compatible Poisson structure having the form
\begin{equation}
\left\lbrace y_i, y_j\right\rbrace = \mathbf{B}_K y_i y_j. 
\end{equation}

Therefore we reach the conclusion that the Poisson structure (induced by the symplectic $2$-form ${\mathcal W_K}$) on the Stokes manifold $(\mathfrak{S}_K, \mathbf{P}_K)$ coincides with the Poisson structure of $(\mathcal{M},\mathbf{B}_K)$, up to a constant multiplicative factor.
\subsection{Flipping the edges}
\label{secflipping}
In the previous section we have established how to define the matrices and the variables $y_j, x_l$ associated to each edge of a given triangulation, in order to get a parametrization of the Stokes matrices. We also computed the Stokes matrices and the Stokes $2$-form for a fixed triangulation, seeing that its matrix coefficient is related to the matrix coefficient of the Poisson structure of the $Y$-cluster manifold of $A_{2K}$-type. \\
 We are now going to show that the $y$--variables associated to two triangulations $T$ and $\tilde{T}$ that are related by a single flip of one of their internal diagonal edges $d_j$,  are related by the rules of the mutation of seed variables (Def. \ref{defmuty}).
Subsequent  flips give different systems of equations for the variables, so we are going to study separately all the possible cases of flip. The equations between the old and the new $y$ variables are obtained by requiring that the Stokes matrices remain the same, independently of the triangulation.\\[5pt]
Consider a generic triangulation of the  $2(K+1)$-gon, and consider any quadrilateral  in the triangulation consisting of two triangles sharing an edge. Since we are considering the case $K\geq2$ we have the following possibilities for the sides of the quadrilateral:
\begin{enumerate}
	\item three sides lie along the perimeter of the polygon, one side is an internal diagonal; 
	\item two sides lie along the perimeter of the polygon and two sides are internal diagonals;
	\item one side is along the perimeter and the three others are internal diagonals;
	\item  all the four sides are internal diagonals.
\end{enumerate}
Notice that the two last cases can occur only for $K>2$. Moreover, the number of $y_j$ variables directly and nontrivially involved in the flip is equal to the number of sides of the quadrilateral that are internal diagonals. We are going to analyze the flip for each case. After the flip, we define some new variables associated to each edge of the new triangulation and we find the corresponding parametrizations of  the Stokes matrices in these new variables. Finally, by imposing the equality between these Stokes matrices, the ones parametrized w.r.t. the first triangulation and the other ones, we obtain an over-determined but compatible system of equations for the old variables and the new ones, $y_j$ and $\tilde{y}_j$. Indeed, notice that the $y_j$ variables are always $2K$ and we have an equation for each Stokes matrix, thus we have a system of $2K+2$ equations in $2K$ variables.
We will see that this system is equivalent to the $y$-mutation correspondent to the vertex on the flipped edge, in the quiver $Q(T)$ associated with the triangulation $T$.  
\paragraph{Case $1$.}
This is the case where three edges of the quadrilateral are along the perimeter. This means that we have only two variables $y$ that are directly and nontrivially involved in the flip. We can suppose that the first vertex, denoted by $v_{2i}$ (supposing that is in even position, the odd case is analogous) have valence only $6$ and that the last one have valence $9$, see Figure \ref{fig: flip case 1}. Every other case can be reduced to this one after an appropriate simplification in the equations we are going to obtain. We denote by $S_j$ the Stokes matrices obtained through the triangulation $T$ and by $\tilde{S}_j$ the ones obtained by the flip of $T$.
\\ 
First, we observe that for every $j\leq2i$ the Stokes matrices are parametrized exactly in the same way w.r.t. the $y_j$ variables and the $\tilde{y}_j$. Thus the equations $S_j(y_k)=\tilde{S}_j(\tilde{y}_k)$ tell us that $y_k=\tilde{y_k}$ for every $k$ that is not incident to $v_{2i},v_{2i+1}, v_{2i+2}$. As a byproduct also the variables $x_l=\tilde{x}_l$ for every $l\leq2i$ they remain invariant.\\
We focus on the equations $S_j(y_k)=\tilde{S}_j(\tilde{y}_k)$ for $k=2i,2i+1,2i+2,2i+3.$ We obtain an over-determined system of four equations from the following four matrix equations
\begin{equation}
    \begin{aligned}
    &D(x_{2i})AV(y_j)^{-1}A V(x_{2i-1}^{-1})^{-1}=D(\tilde{x}_{2i})AV(\tilde{y}_{j+1})^{-1}A V(\tilde{y}_{j+1})^{-1}A V(x_{2i-1}^{-1})^{-1}\\
    \\
    &V(x_{2i+1}^{-1}) A V(y_{j+1})AD(x_{2i})^{-1} = V(\tilde{x}_{2i+1}^{-1}) A D(\tilde{x}_{2i})^{-1}\\
    \\ 
    & D(x_{2i+2})A V(x_{2i+1}^{-1})^{-1} = D(\tilde{x}_{2i+1}^{-1}) AV(\tilde{y}_{j+1})^{-1}A V(\tilde{x}_{2i+2})^{-1}\\
    \\
    & V(x_{2i+3}^{-1})A V(y_{j-1})^{-1}AV(y_j)AV(y_{j+1})^{-1}AD(x_{2i+2})^{-1} = V(\tilde{x}_{2i+3}^{-1}) AV(\tilde{y}_{j-1})AV(\tilde{y}_{j})AD(\tilde{x}_{2i+2})^{-1}
    \end{aligned}
\end{equation}

 It follows then the following relations between the old and the new variables must hold
\begin{equation}
\label{eq: mutation case 1}
\tilde{y}_j^2=(1+y_{j+1}^2)y_j^2,\;\;\;\tilde{y}_{j+1}^2= \frac{1}{y_{j+1}^2}
\end{equation}
where $y_j$ is the variable on the diagonal $v_{2i}-v_{2i+3}$ and $y_{j+1}$ is the one on the diagonal $v_{2i+1}-v_{2i+3}$ as show in Figure \ref{fig: flip case 1}. One obtains these results from the second and third equation directly, then the other equations are automatically satisfied replacing these relations.  

\begin{figure}
    \centering
    
 \includegraphics[width=0.7\textwidth]{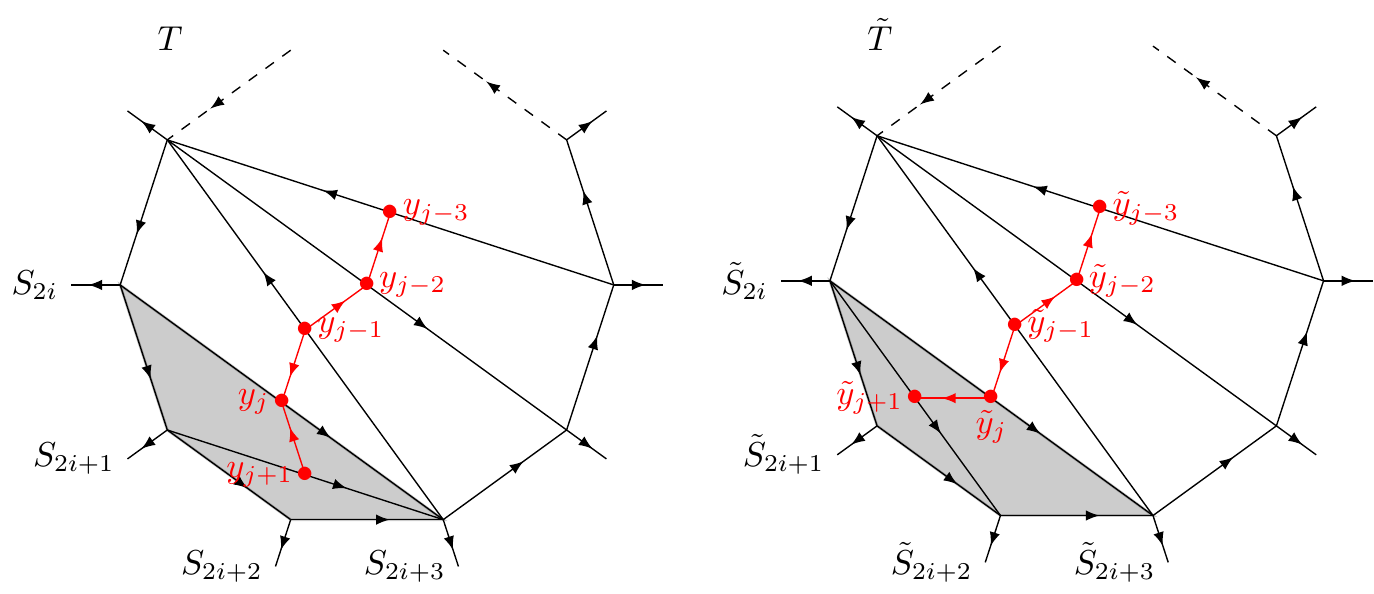}   
    \caption{A flip of a quadrilateral inside the triangulation $T$ with $3$ sides along the perimeter of the polygon and the new triangulation $\tilde{T}$ obtained in this way.}
    \label{fig: flip case 1}
\end{figure}
\paragraph{Case $2$.}
Here we consider the case where there are two edges of the quadrilateral on the perimeter of the polygon, and the other two edges are internal diagonals.  We can suppose as before that the first vertex is even $v_{2i}$. Also, we can assume that $v_{2i}, v_{2i+4}$ both have valence $8$ and $v_{2i+3}$ has valence $4$. Then all the other cases (when the valences of these vertices are higher) can be reduced to this one, after appropriate simplification. In this case three variables $y$  are directly involved in the flip. Indeed, by the fact that $S_j(y_k)=\tilde{S}_j(\tilde{y}_k)$ for every $j$, we obtain that $y_l=\tilde{y}_l$ for any index $l$ that is not incident to $v_{2i}, v_{2i+1}, v_{2i+2}, v_{2i+3}$ and also for all the variables that stay on the right of the $y_j$ diagonal, see Figure \ref{fig: flip case 2}.
Furthermore, by looking at $j=2i,2i+1,2i+2,2i+3$ we obtain the following over-determined system of four equations, from the four matrix equations
\begin{equation}
  \begin{aligned}
    &D(x_{2i})AV(y_{j+1})AV(y_j)^{-1}A V(x_{2i-1}^{-1})^{-1}=D(\tilde{x}_{2i})AV(\tilde{y}_{j})A  V(\tilde{x}_{2i-1}^{-1})^{-1}\\
    \\
    &V(x_{2i+1}^{-1}) A V(y_{j+2})AD(x_{2i})^{-1} = V(\tilde{x}_{2i+1}^{-1}) AV(\tilde{y}_{j+2})AV(\tilde{y}_{j+1})A D(\tilde{x}_{2i})^{-1}\\
    \\ 
    & D(x_{2i+2})AV(x_{2i+1})^{-1}= D(\tilde{x}_{2i+2}) A V(\tilde{x}_{2i+1})^{-1}\\    
    \\
    & V(x_{2i+3}^{-1})A V(y_{j+1})^{-1}AV(y_{j+2})^{-1}A D(x_{2i+2})^{-1} = V(\tilde{x}_{2i+3}^{-1}) A V(\tilde{y}_{j+2})AD(\tilde{x}_{2i+2})^{-1}.
    \end{aligned}   
\end{equation}
In particular, from the first three equations we obtain the following relations between the old and the new variables
\begin{equation}
\label{eq: mutation case 2}
    \tilde{y}_j^2= y_j^2\frac{y_{j+1}^2}{1+y_{j+1}^2},\;\;\; \tilde{y}_{j+1}^2=\frac{1}{y_{j+1}^2},\;\;\;\tilde{y}_{j+2}^2=y_{j+2}^2(1+y_{j+1}^2),
\end{equation}
and all the other equations are then satisfied by replacing these quantities (included the equation for $j=2i+4$).

\begin{figure}
    \centering
    \includegraphics[width=0.7\textwidth]{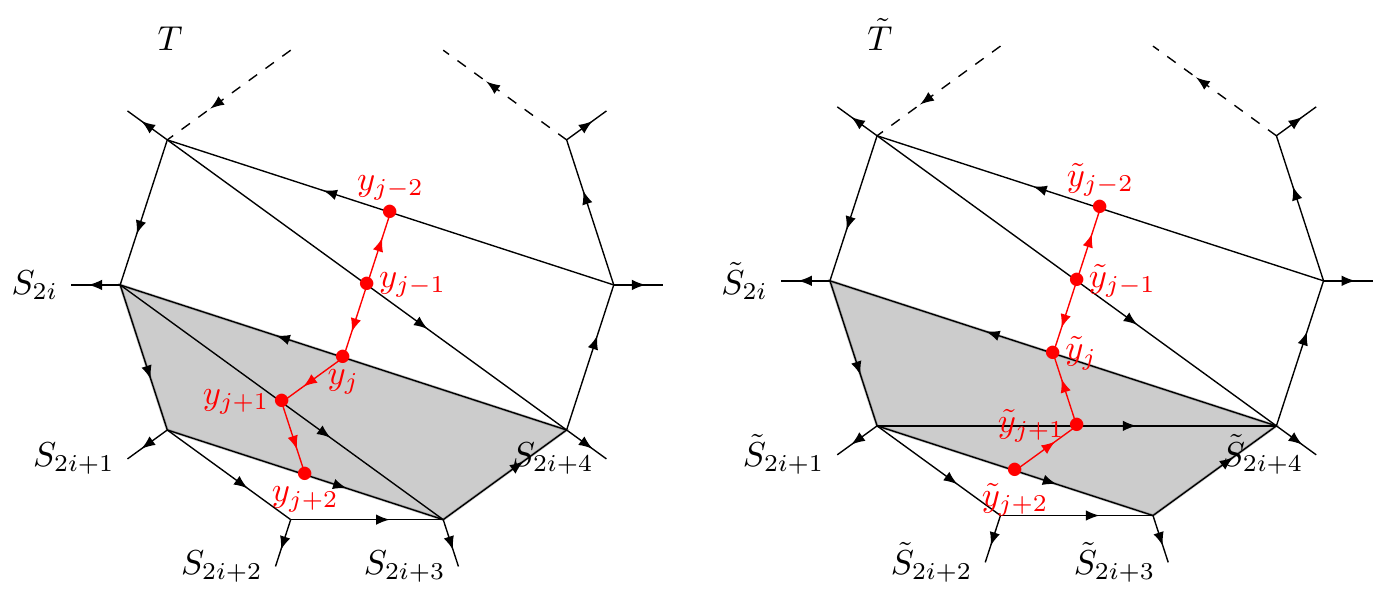}
    \caption{A flip of a quadrilateral inside the triangulation $T$ with $2$ sides along the perimeter of the polygon and the new triangulation $\tilde{T}$ obtained in this way.}
    \label{fig: flip case 2}
\end{figure}
	
\paragraph{Case $3$.}
Here we consider the case where three edges of the quadrilateral are internal diagonals of the polygon and only one edge is on its perimeter. Notice that this means that there are four variables $y$ that are nontrivially involved in the flip. We suppose as before that the first edge considered is even $v_{2i}$ and that all the vertices involved in the quadrilateral and their adjacent vertices have minimal valence, as in Figure \ref{fig: flip case 3}. As in the previous cases, the equations $S_l(y_k)=\tilde{S}_l(\tilde{y}k)$ for the indices $l\neq 2i, \dots, 2i+4$ give that the variables $y_k=\tilde{y}_k$ for the $k$ that are not incident to the vertices $v_{ 2i}, \dots,v_{ 2i+4}$. Then looking at the matrix  equations for $l=2i,\dots2i+4$ we have the four matrix equations
\begin{equation}
\begin{aligned}
   & D(\tilde{x}_{2i})AV(\tilde{y}_j)AV(\tilde{x}_{2i-1}^{-1})= D(x_{2i})AV(y_{j+1})AV(y_{j})^{-1}AV(x_{2i-1}^{-1})\\
   \\
   & V(\tilde{x}_{2i+1}) A V(\tilde{y}_{j+2})AV(\tilde{y}_{j+1})AD(\tilde{x}_{2i})^{-1} = V(x_{2i+1}^{-1})AV(y_{j+2})AD(x_{2i})^{-1}\\
   \\
   &D(\tilde{x}_{2i+2})AV(\tilde{x}_{2i+1}^{-1})^{-1} = D(x_{2i+2})AV(x_{2i+1}^{-1})^{-1}\\
   \\
   &V(\tilde{x}_{2i+3}^{-1})AV(\tilde{y}_{j+3})^{-1}AV(\tilde{y}_{j+2})^{-1}AD(\tilde{x}_{2i+2})^{-1}=V(x_{2i+3}^{-1})AV(y_{j+3})^{-1}A V(y_{j+1})^{-1}AV(y_{j+2})^{-1}AD(x_{2i+2})^{-1}.
\end{aligned}
\end{equation}
From these equations we obtain that the old variables and the new variables are related through the following  relations 
\begin{equation}
\label{eq: mutation case 3}
    \tilde{y}_j^2= y_{j+1}^2\frac{y_j^2}{1+y_{j+1}^2},\;\;\; \tilde{y}_{j+1}^2=\frac{1}{y_{j+1}^2},\;\;\; \tilde{y}_{j+2}^2= y_{j+2}^2\frac{y_{j+1}^2}{1+y_{j+1}^2},\;\;\; \tilde{y}_{j+3}^2= y_{j+3}^2(1+y_{j+1}^2)
\end{equation}
and all the other equations (included for the vertices $v_{2i+4},v_{2i+5}$) are identically satisfied once we replace the relations above.
\begin{figure}
    \centering

    \includegraphics[width=0.7\textwidth]{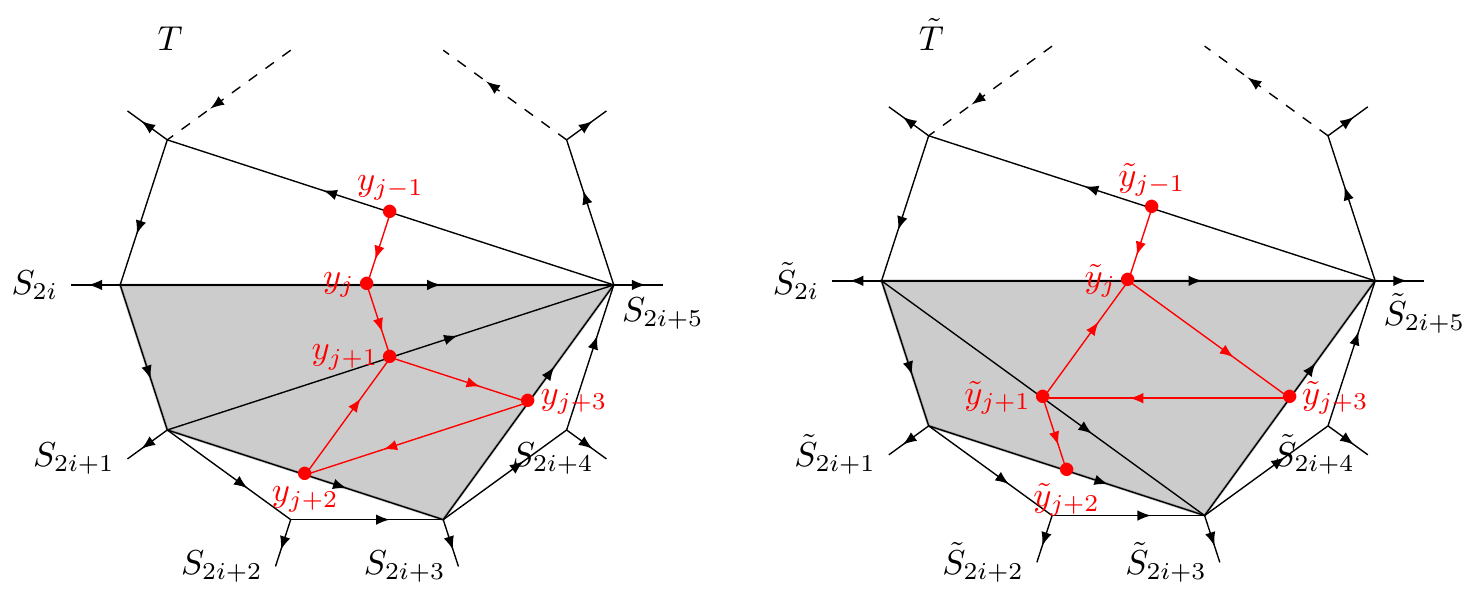}
    \caption{A flip of a quadrilateral inside the triangulation $T$ with only $1$ side along the perimeter of the polygon and the new triangulation $\tilde{T}$ obtained in this way.}
    \label{fig: flip case 3}
\end{figure}
	\paragraph{Case $4$.} Here we consider the case where all the sides of the quadrilateral are internal diagonals. We suppose, as always, to have the first vertex that is even $v_{2i}$ and that each vertex has minimal valence, as in Figure \ref{fig: flip case 4}. Every other case, with higher order valence for the vertices involved, can be reduced to this one after appropriate simplification. In this case, we have five variables $y$ directly involved in the flip, thus we will have one more equation than in the other cases.\\
	By looking at the equations $S_l(y_k)=\tilde{S}_l(\tilde{y}_k)$ for $l\neq 2i,\dots, 2i+5$, we get that $y_k=\tilde{y}_k$ for every index $k$ that is not adjacent to the flipped edge with coordinate $y_{j+4}$. Then by looking at the equations for $l=2i,\dots, 2i+4$ we have the following matrix-valued system
	\begin{equation}
	    \begin{aligned}
	        &D(x_{2i})AV(y_j)AV(y_{j+1})^{-1}AV(y_{j+3})AV(x_{2i-1})^{-1}=D(\tilde{x}_{2i})AV(\tilde{y}_j)AV(\tilde{y}_{j+3})AV(\tilde{x}_{2i-1})^{-1}\\
	        \\
	        &V(x_{2i+1})AD(x_{2i})^{-1}=V(\tilde{x}_{2i+1})AD(\tilde{x}_{2i})^{-1}\\
	        \\
	        &D(x_{2i+2})AV(y_{j+1})AV(y_{j})^{-1}AV(x_{2i+1})^{-1}=D(\tilde{x}_{2i+2})AV(\tilde{y}_{j+1})AV(\tilde{y}_{j+4})AV(\tilde{y}_{j})^{-1}AV(\tilde{x}_{2i-1})^{-1}\\
	        \\
	        &V(x_{2i+3})AD(x_{2i+2})^{-1}=V(\tilde{x}_{2i+3})AD(\tilde{x}_{2i+2})^{-1}\\
	        \\
	        & D(x_{2i+4})AV(y_{j+2})^{-1}AV(y_{j+4})AV(y_{j+1})^{-1}AV(x_{2i+3})^{-1}=D(\tilde{x}_{2i+2})AV(\tilde{y}_{j+2})AV(\tilde{y}_{j+1})^{-1}AV(\tilde{x}_{2i+3})^{-1}.
	    \end{aligned}
	\end{equation}
	This system is solved through the following relations between the old and the new variables 
	\begin{equation}
	\label{eq: mutation case 4}
	   \tilde{y}_j^2=y_j^2(1+y_{j+4}^2),\;\;\; \tilde{y}_{j+1}^2=y_{j+1}^2\frac{y_{j+4}^2}{1+y_{j+4}^2},\;\;\;\tilde{y}_{j+2}^2=y_{j+2}^2(1+y_{j+4}^2),\;\;\; \tilde{y}_{j+3}^2=y_{j+3}^2\frac{y_{j+4}^2}{1+y_{j+4}^2}, \;\;\; \tilde{y}_{j+4}^2=\frac{1}{y_{j+4}^2}
	\end{equation}
	and they also satisfy the equations for $l=2i+5,2i+6.$
\begin{figure}
    \centering

    \includegraphics[width=0.7\textwidth]{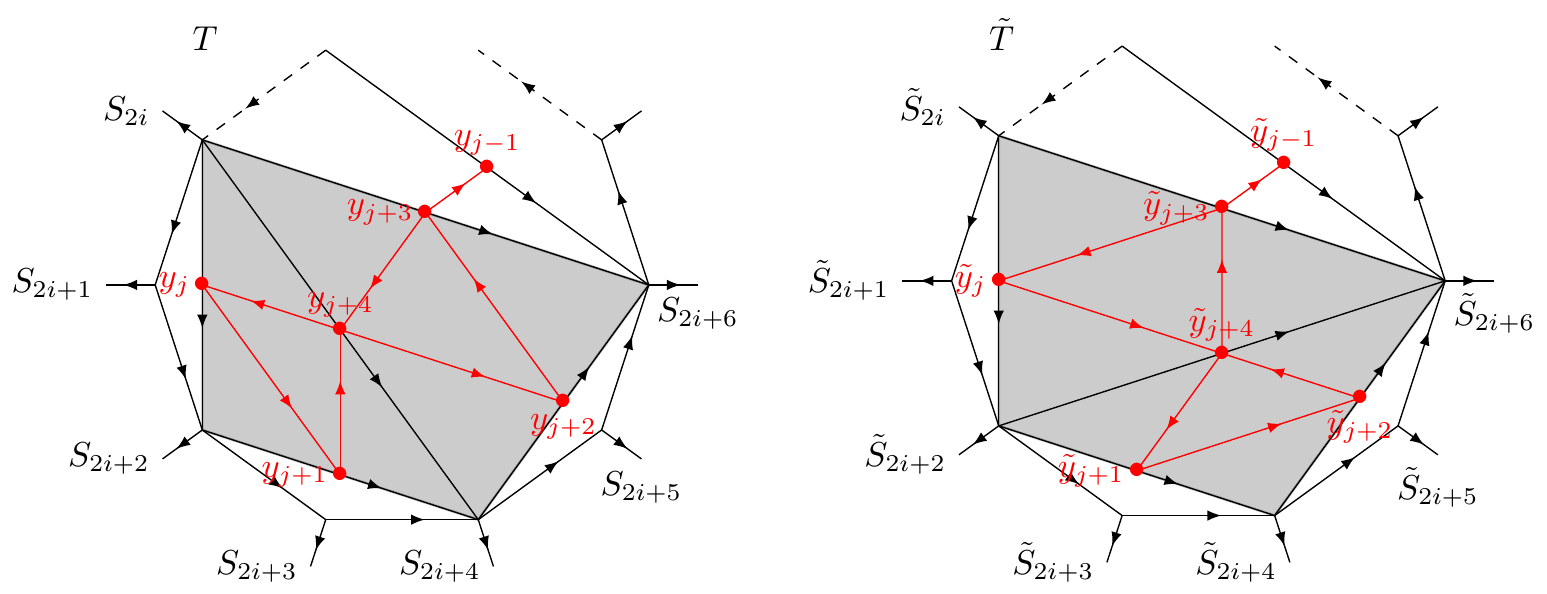}
    \caption{A flip of a quadrilateral inside the triangulation $T$ with no sides along the perimeter of the polygon and the new triangulation $\tilde{T}$ obtained in this way.}
    \label{fig: flip case 4}
\end{figure}

We observe that in each case we obtained that the system of equations for the old and new $y$ variables obtained from the matrix equations $S_l(y_k)=\tilde{S}_l(\tilde{y}_k)$ is solved by some  $y$-mutation relations of the Dynkin diagram of $A_{2K}$-type. In particular, every set of equations \eqref{eq: mutation case 1}, \eqref{eq: mutation case 2}, \eqref{eq: mutation case 3}, \eqref{eq: mutation case 4} coincide with the $y$-mutation w.r.t. the vertex $y_l$ associated to the flipped edge of the triangulation $T$ of the polygon, of the Dynkin diagram of $A_{2K}$-type associated to the triangulation $T$ for the square of its variables. 
\subsection{Example: the case $K=2$}	
We work out on the case $K=2$, i.e. the case of the hexagon.
In particular, we are going to take the fixed triangulation $T_0$ of the hexagon (e.g. the one in Figure \ref{fig:triang and dynk}), and we consider the variables and the matrices associated to each edge of the graph in the common way explained before. We compute then the Stokes matrices and the Stokes $2$-form ${\mathcal W_2}$ in these variables.\\ 
Then, we consider all the possible flip of this triangulation, w.r.t. the edges with variables $y_2,y_3,y_4$ as in Figure \ref{fig:4 triangu K=2}, and we perform the same computations above with the new variables associated to each new triangulation obtained in that way. We will see that in each case, the inverse of the matrix coefficient of the Stokes $2$-form is, up to the same factor $-\frac{1}{4}$ the adjacency matrix of a certain mutation of the $A_4$ Dynkin diagram, the one given in Figure \ref{fig:dynkdiag}. 

\begin{figure}
	\centering

\includegraphics[width=0.6\textwidth]{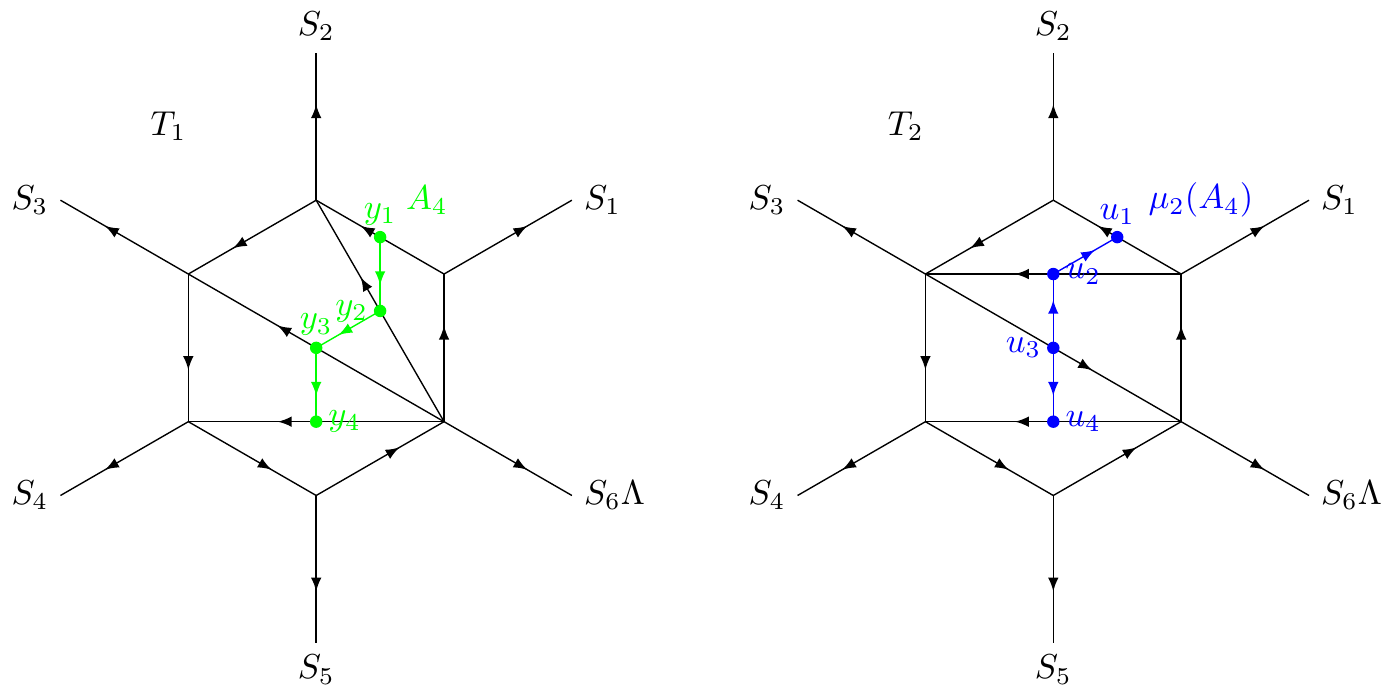}\\
\includegraphics[width=0.6\textwidth]{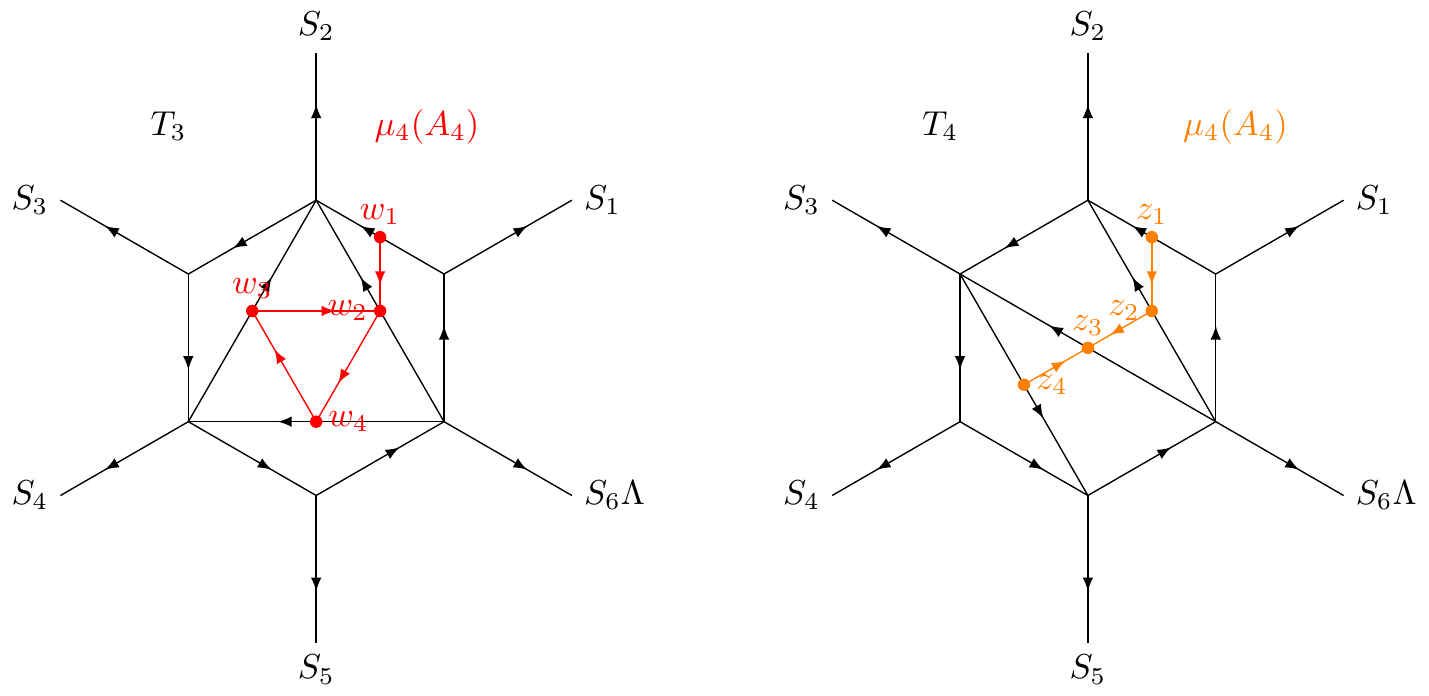}

	\caption{The 4 triangulations considered are $T_1$ and then all the others obtained from $T_1$ by a flip of one of the diagonals $d_j$ for $j=2,3,4.$}
	\label{fig:4 triangu K=2}
\end{figure}
\begin{itemize}
\item For the triangulation $T_1$ the variables $x_l$ are
\begin{equation}
x_2=y_1y_2^{-1},\;\; x_3 = y_1y_2^{-1}y_3, \;\; x_4= y_1y_2^{-1}y_3y_4^{-1},\;\; x_5=x_4, \;\; x_6=y_1.
\end{equation}
The $2$-form ${\mathcal W_2}^{T_1}$ is log-canonical in the variables $y_i$ and such that its matrix coefficient has inverse 
\begin{equation}
\mathbf{P}_2 ^{T_1}=\frac{1}{4}
\begin{pmatrix}
0&1&0&0\\
-1&0&1&0\\
0&-1&0&1\\
0&0&-1&0
\end{pmatrix} =\frac{1}{4} Adj_{A_4}.
\end{equation}
\item For the triangulation $T_2$ the variables $x_l$ are
\begin{equation}
\begin{aligned}
& x_2 = u_1, \;\; x_3=  u_1 u_2 u_3,\;\;  x_4= u_1u_2 u_3u_4^{-1},\;\;x_5=u_1 u_2 u_3u_4^{-1},\;\;x_6 = u_1u_2^{-1}.
\end{aligned}
\end{equation}
The $2$-form ${\mathcal W_2}^{T_2}$ is log-canonical in the variables $y_i$ and such that the inverse of its coefficient matrix, namely $\mathbf{P}_2^{T_2}$ gives 
\[\mathbf{P}_2 ^{T_2}=\frac{1}{4}
\begin{pmatrix}
0&-1&0&0\\
1&0&-1&0\\
0&1&0&1\\
0&0&-1&0
\end{pmatrix} =\frac{1}{4} Adj_{\mu_2(A_4)}.\]
\item For the triangulation $T_3$ the variables $x_l$ are
\begin{equation}
x_2=w_1w_2^{-1}w_3^{-1},\;\; x_3=x2, \;\;x_4 =  w_1w_2^{-1}w_3^{-2}w_4^{-1},\;\; x_5=x_4,\;\; x_6=w_1.
\end{equation}
The $2$-form ${\mathcal W_2}^{T_3}$ is such that the inverse of its coefficient matrix, namely $\mathbf{P}_2^{T_3}$ gives 
\[\mathbf{P}_2 ^{T_3}=\frac{1}{4}
\begin{pmatrix}
0&1&0&0\\
-1&0&-1&1\\
0&1&0&-1\\
0&-1&1&0
\end{pmatrix} =\frac{1}{4} Adj_{\mu_3(A_4)}.\]
\item  For the triangulation $T_4$ the variables $x_l$ are
\begin{equation}
x_2=t_1t_{2}^{-1},\;\; x_3=t_1t_{2}^{-1} t_3t_4,\;\; x_3=t_4,\;\; x_5= t_1t_{2}^{-1} t_3t_4^2.
\end{equation}
The $2$-form ${\mathcal W_2}^{T_4}$ is such that the inverse of its coefficient matrix, namely $\mathbf{P}_2^{T_4}$ gives 
\[\mathbf{P}_2 ^{T_4}=\frac{1}{4}
\begin{pmatrix}
0&1&0&0\\
-1&0&1&0\\
0&-1&0&-1\\
0&0&1&0
\end{pmatrix} =\frac{1}{4} Adj_{\mu_4(A_4)}.\]
\end{itemize}
Furthermore the equations $S_i(\vec{y})=\tilde{S}_i(\vec{u})$ that impose the Stokes equations parametrized in the 2 triangulations $T_1$ and $T_j$ to be equal, give exactly that $u_i^2, w_i^2$ or $t_i^2$ respectively for $j=2,3,4$ are $y$-mutation of $y_i^2$ related to $A_4$ w.r.t. the vertices $y_2, y_3, y_4$. 
\subsection{Computation of the Poisson brackets for the original monodromy parameters}
In the previous sections we have parametrized the Stokes manifold $\mathfrak{S}_K$  of dimension $2K$, by using the variables $y_j$ for $j=1,\dots,2K$ of the $A_{2K}$ cluster algebra type. Using this parametrization, explicitly computed in Lemma \ref{prop:1}, we also proved that the two-form ${\mathcal W_K}$ defined on $\mathfrak{S}_K$ is symplectic and that the variables $y_j$ are log-canonical for this two form. We also computed the Poisson brackets $\mathbf{P}_K$ induced by the symplectic structure ${\mathcal W_K}$ on $\mathfrak{S}_K$. Now, we want to compute these Poisson brackets $\mathbf{P}_k$ on the parametrization of the original monodromy parameters $s_j,$ for $j=1,\dots,2k+2$ and $\lambda$ describing $\mathfrak{S}_K.$
In particular, we would like to show that the Poisson brackets $\mathbf{P}_K$ for the $y_j$ defined in \eqref{eq:poisson str} are a log-canonical formulation of the following bracket.
\begin{defn}
\label{FNPB}
Consider the nonlinear Poisson bracket on $\mathbb{C}^{2K+2}\times \mathbb{C}^*$ with coordinates $(s_1,\dots, s_{2L},\lambda)$ given by 
\begin{equation}
\begin{aligned}
\label{eq:PB_K fl newell} 
\Big\{s_j,s_l \Big\}_{_{FN}} & =  \delta_{j,l-1} - \frac{\delta_{j,1}\delta_{l, 2K+2}} {\lambda^{2}}  + (-1)^{j-l+1} s_j s_l, \qquad j<l.\cr\\
\Big\{ s_j,\lambda \Big\}_{_{FN}} &=(-1)^{j}  s_j\lambda.
\end{aligned}
\end{equation}
\end{defn}
These Poisson structure first appeared in \cite{FlaschkaNewell} (see section $3,5$).

\begin{prop}
\label{propideal}
Let 
\be
F= F_K= \begin{pmatrix}
1&s_{1}\\
0&1
\end{pmatrix} 
\begin{pmatrix}
1&0\\
s_{2}&1
\end{pmatrix} \dots
\begin{pmatrix}
1&s_{2K+1}\\
0&1
\end{pmatrix}
\begin{pmatrix}
1&0\\
s_{2K+2}&1
\end{pmatrix} \lambda^{\s_3}.
\ee
Let $\s_3,\s_+, s_-$ be the matrices 
\be
\s_3 = \begin{pmatrix}
1&0\\
0&-1
\end{pmatrix}
,\ \s_+ = \begin{pmatrix}
0&1\\
0&0
\end{pmatrix},\ \s_- = \begin{pmatrix}
0&0\\
1&0
\end{pmatrix}
\ee
{\bf(1)} The matrix  $F$ satisfies
\be
\label{ideal}
\{s_1,F\}_{_{FN}} &= \frac {s_1}2[\s_3,F] + [\s_-,F]\cr
\{s_{2K+2}, F\}_{_{FN}} &= \frac {s_{2K+2}}2 [F, \s_3]+ \frac 1{\l^2}[\s_+,F]\cr
\{s_\ell, F\}_{_{FN}} &= (-1)^\ell[F, \s_3], \ \ \ 2\leq \ell\leq 2k+1\cr
\{\l, F\}_{_{FN}} & = \frac 1 2[\s_3,F].
\ee
{\bf (2)} The unique Casimir function for the bracket \eqref{PB_K} is $\mathfrak C= \Tr (F)$; \\
{\bf (3)} The  sub-varieties $\S_K= \{F_{K}= \Id \}$ are Poisson sub-varieties.
\end{prop}
We defer the proof to  Appendix \ref{Appideal}.

\begin{thm}
\label{thmmain}
The parametrization given in Lemma \ref{prop:1} for the Stokes parameters $s_j, j=1,\dots, 2K+2$ and the formal monodromy exponent $\lambda$ transforms the Poisson bracket \eqref{eq:PB_K fl newell} in the bracket \eqref{eq:poisson str}.
\end{thm}
\begin{proof}
We start by observing that the bracket \eqref{eq:poisson str} is such that all even-indexed variables commute amongst themselves, and so do the odd ones. 
We now verify that the bracket \eqref{eq:poisson str} yields the bracket \eqref{eq:PB_K fl newell} under the map \eqref{eq:param stokes param y}.
We will verify some of the brackets explicitly and leave the rest of the verification to the reader.
Let us start with the  case $\{s_{2k+1},\lambda\}$ for $k<K$: since $\lambda$ is a function of only the even variables it commutes with the even ones and we can write 
\begin{equation}
\{s_{2k+1} ,\lambda \} = -\prod_{j=1}^{k} y_{2j}^2  \left\{ \prod_{j=0}^{k-1} y_{2j+1}^{-2} + \prod_{j=0}^{k} y_{2j+1}^{-2} , \prod_{j=1}^{K} y_{2j}^2\right\}.
\end{equation}
This computation is easily done by passing to the logarithms of the variables $y_j$'s, in which the Poisson bracket \eqref{eq:poisson str} is constant: thus both terms inside the bracket are log--canonical.  Then one observes that the bracket above involves a telescopic sum and only the term $y_{1}$ yields a contribution and we obtain 
\begin{equation}
\{s_{2k+1} ,\lambda\} ={-} s_{2k+1}\lambda.
\end{equation}
The case $\{s_{2K+1},\lambda\}$ is handled similarly. 
Consider now an even variable $s_{2k}$ for $k<K$; since $\lambda$ is a function of only the even variables we can write 
\begin{equation}
\{ s_{2k}, \lambda\} = \frac{(1+y_{2k}^2)}{ \prod_{j=1}^{ k} y_{2j}^2 }\left\{  \prod_{j=0}^{k-1} y_{2j+1}^2 ,\prod_{j=1}^{K} y_{2j}^2\right\} =  s_{2k}\lambda,
\end{equation}
where we have used the same telescopic-sum argument. Again, the case $\{s_{2K+2},\lambda \}$ is handled similarly observing that $s_{2K+2} = y_1^2$ times a function of only even variables. 

Let us now consider the bracket $\{s_a, s_b\}$; suppose both $a=2k,b=2l$ are even. 
\begin{equation}
\label{114}
\begin{aligned}
\left\{
\frac{(1+ y_{2k}^2)}
{\prod_{j=1}^k y_{2j}^2}  {\prod_{j=1}^{k}y_{2j-1}^2}, 
\frac{(1+ y_{2l}^2)}
{\prod_{j=1}^l y_{2j}^2}  {\prod_{j=1}^{l}y_{2j-1}^2}\right\}
&=\frac{(1+ y_{2k}^2)}
{\prod_{j=1}^k y_{2j}^2}\left\{
  {\prod_{j=1}^{k}y_{2j-1}^2}, 
\frac{(1+ y_{2l}^2)}
{\prod_{j=1}^l y_{2j}^2} \right\} {\prod_{j=1}^{l}y_{2j-1}^2} 
\\ &+
{\prod_{j=1}^{k}y_{2j-1}^2}\left\{
\frac{(1+ y_{2k}^2)}
{\prod_{j=1}^k y_{2j}^2}  , 
 {\prod_{j=1}^{l}y_{2j-1}^2}\right\}\frac{(1+ y_{2l}^2)}
{\prod_{j=1}^l y_{2j}^2} 
\end{aligned}
\end{equation}
The computation relies on the following simple observation, which can be used for both terms by interchanging the roles of $k$ and $l$:
\begin{equation}
\left\{
  {\prod_{j=1}^{k}y_{2j-1}^2}, 
\frac{1}
{\prod_{j=1}^l y_{2j}^2} \right\}  = 
\begin{cases}
{-}\frac{{\prod_{j=1}^{k}y_{2j-1}^2} }
{\prod_{j=1}^l y_{2j}^2}   &  k \leq  l \\[30pt]
 0 & k> l.
 \end{cases}.
 \label{115}
\end{equation}
Now let $k\leq l-1$: then the second bracket in \eqref{114} is zero and the first yields back $s_{2k} s_{2l}$ which is consistent with \eqref{eq:PB_K fl newell}.
The odd-odd case is similarly handled.\\
We still have to check the case even-odd. For that, consider the case $\{s_{2k}, s_{2l+1}\}$ for $k\leq l$:
\begin{equation}
\begin{aligned}
&
\left\{
\frac{(1+ y_{2k}^2)}
{\prod_{j=1}^k y_{2j}^2}  {\prod_{j=1}^{k}y_{2j-1}^2}, 
- \frac{(1+ y_{2l+1}^2)}
{\prod_{j=0}^l y_{2j+1}^2}  {\prod_{j=1}^{l}y_{2j}^2}\right\}
=\\
&=-
\frac{(1+ y_{2k}^2)}
{\prod_{j=1}^k y_{2j}^2} \left\{
 {\prod_{j=1}^{k}y_{2j-1}^2}, 
{\prod_{j=1}^{l}y_{2j}^2}\right\} \frac{(1+ y_{2l+1}^2)}
{\prod_{j=1}^l y_{2j+1}^2}  
-
{\prod_{j=1}^{k}y_{2j-1}^2}\left\{
\frac{(1+ y_{2k}^2)}
{\prod_{j=1}^k y_{2j}^2}  , 
 \frac{(1+ y_{2l+1}^2)}
{\prod_{j=0}^l y_{2j+1}^2}\right\}  {\prod_{j=1}^{l}y_{2j}^2}
\label{117}
\end{aligned}
\end{equation}
The first bracket in \eqref{117} gives $  
 {\prod_{j=1}^{k}y_{2j-1}^2} 
{\prod_{j=1}^{l}y_{2j}^2}$ 
and hence 
\begin{equation}
\{s_{2k}, s_{2l+1}\} = s_{2k} s_{2l+1} -
{\prod_{j=1}^{k}y_{2j-1}^2}\left\{
\frac{(1+ y_{2k}^2)}
{\prod_{j=1}^k y_{2j}^2}  , 
 \frac{(1+ y_{2l+1}^2)}
{\prod_{j=0}^l y_{2j+1}^2}\right\}  {\prod_{j=1}^{l}y_{2j}^2}
\label{118}
\end{equation}
The several contributions in \eqref{118} can all be accounted for by the formula \eqref{115}: if $l \geq k+1$ then one sees immediately that all terms in the bracket in \eqref{118} vanish. The only case when the bracket gives a nonzero contribution is for $k=l$:
\begin{equation}
\left\{
\frac{(1+ y_{2k}^2)}
{\prod_{j=1}^k y_{2j}^2}  , 
 \frac{(1+ y_{2k+1}^2)}
{\prod_{j=0}^k y_{2j+1}^2}\right\} =\left\{
\frac{1}
{\prod_{j=1}^k y_{2j}^2}  , 
 \frac{1}
{\prod_{j=0}^{k-1} y_{2j+1}^2}\right\}  
={-}
\frac 1{\prod_{j=1}^k y_{2j}^2}  
 \frac{1}
{\prod_{j=0}^{k-1} y_{2j+1}^2} .
\end{equation}
Combining this with \eqref{118} gives finally 
\begin{equation}
\{s_{2k} , s_{2l+1}\} = 1 {+} s_{2k} s_{2l+1}
\end{equation}
To complete the verification remains only to check the case 
\begin{equation}
\begin{aligned}
\{s_1&, s_{2K + 2}\} = \left\{ -y_1^{-2},\sum_{l=1}^{K} \frac{y_1^2}{\prod_{j=1}^{K} y_{2j}^{2} \prod_{j=l}^K y_{2j}^2}   \right\}
=
-\sum_{l=1}^Ky_1^2 \left\{ y_1^{-2}, \frac{1}{\prod_{j=1}^{K} y_{2j}^{2} \prod_{j=l}^K y_{2j}^2}   \right\}
=
\\
&=
-\sum_{l=1}^Ky_1^2 \left\{ y_1^{-2}, \frac{1}{\prod_{j=1}^{K} y_{2j}^{2} }   \right\} \frac 1{\prod_{j=l}^K y_{2j}^2} -\sum_{l=1}^Ky_1^2 \left\{ y_1^{-2}, \frac{1}{\prod_{j=l}^{K} y_{2j}^{2} }   \right\} \frac 1{\prod_{j=1}^K y_{2j}^2}.
\end{aligned}
\end{equation}
In the second sum only the term $l=1$ contributes  and the result of this is $\frac 1 {\lambda^2}$; the first sum instead contributes $-s_1 s_{2K+2}$ and in total we find 
\begin{equation}
\{s_1, s_{2K + 2}\} = {-}\frac 1{\lambda^2} {+}  s_1 s_{2K+2}.
\end{equation}
The verification is thus complete.
\end{proof}
%

\section{Log canonical coordinates for the Ugaglia bracket}
\label{secUgaglia}
In this section we show, without detailed proofs, how to construct the log--canonical coordinates for the Ugaglia bracket \cite{Ugaglia} using the same idea exploited in the first half of the paper. 

We remind that this is a Poisson bracket on the Stokes' manifold for the following ODE:
\be
\label{Frobenius}
\Psi'(z) = \le( U + \frac V z\ri) \Psi(z), \ \ U\in \CSA, \ \ V=-V^t \in so(n). 
\ee
For simplicity of exposition assume $U= {\rm diag} (u_1,\dots, u_n)$ such that $\Re u_j<\Re u_{j+1}$ so that the Stokes' rays can be chosen as $\mathbb R_\pm$, which we take oriented towards $\infty$. 

Because of the symmetry, the Stokes' matrix, $S = S_+$,  on $\R_+$  is upper triangular, with unit on the diagonal and  on $\R_-$ the Stokes' matrix is $S_- = S_+^{-t}$. 
The formal asymptotic of $\Psi$ has vanishing exponents of formal monodromy;
\be
\Psi_{form} = \wh Y_\infty(z) {\rm e}^{z U}, \ \ \ \wh Y_\infty(z) = \Id + \sum_{j\geq 1} \frac {Y_j} {z^j},\ \ \ z\to \infty.
\ee
The equation \eqref{Frobenius} has a Fuchsian singularity at $z=0$ and the monodromy matrix  is $M = S^{-t} S$. If $L$ denotes the diagonal matrix of eigenvalues of the matrix $V$ in \eqref{Frobenius}, then  $M = C {\rm e}^{2i\pi L} C^{-1}$, for some matrix $C\in SL_n(\C)$ called the {\em connection matrix}. 
Then one has a solution defined in the universal cover of a punctured disk around the origin of the form 
\be
\Psi_0(z) = G_0\wh Y_0(z) z^{L} C^{-1}, \ \ \ \ \wh Y_0(z) = \Id +\sum_{j\geq 1} H_j z^j. 
\ee
where $L =  G_0^{-1} V G_0$ and the series $\wh Y_0(z)$ has infinite radius of convergence.
The Ugaglia Poisson bracket is given by the following set of equations (we change the normalization  relative to loc. cit. so that the Ugaglia Poisson bracket is this Poisson bracket multiplied by $i\pi/2$)\footnote{In the formula (3.2f) of the  published paper \cite{Ugaglia}  the last is index is printed as $kj$ but it is clearly a typographical error.}
\be
\label{UgP}
&\{s_{ik} , s_{i\ell}\}_{_{U}} = 2s_{k\ell} - s_{ik}s_{i\ell} & i<k<\ell\\
&\{s_{ik} , s_{jk}\}_{_{U}} = 2s_{ij} - s_{ik}s_{jk} & i<j<k\\
&\{s_{ik} , s_{k\ell}\}_{_{U}} =  s_{ik}s_{k\ell} -2s_{i\ell}  & i<k<\ell\\
&\{s_{ik} , s_{j\ell}\}_{_{U}} = 0 & i<k<j<\ell\\
&\{s_{ik} , s_{j\ell}\}_{_{U}} =0& i<j<\ell<k\\
&\{s_{ik} , s_{j\ell}\}_{_{U}} = 2(s_{ij}s_{k\ell} - s_{i\ell} s_{jk})& i<j<k<\ell.
\ee
As written, its Casimirs are the exponents of the eigenvalues of monodromy, namely $L$; notice that there are $\lfloor\frac n 2 \rfloor$ independent such eigenvalues because $V=-V^t$ and hence if $\l$ is an eigenvalue of $V$ then so is $-\l$.
\begin{figure}

\begin{center}
\begin{tikzpicture}[scale=1.85]
\coordinate (pr) at (1.5,0);
\coordinate (pl) at (-1.5,0);
\coordinate (nu) at (0,1.4);
\coordinate (nd) at (0,-1.2);
\coordinate (t) at (-0.6,0.6);
\coordinate (b) at (t);
\draw [directed] (pr) to  node[above] {$S$} (3,0);
\draw [directed] (pl) to node [above] {$S^{-t}$} (-3,0);
\draw [directed] (pr) to  [out=120, in=0]node[above, pos=1] {$Q$} (0,1.5) to [out=180, in=60] (pl);
\draw [directed] (pl) to  [out=-60, in=180]node[below, pos=1] {$Q^{-t}$} (0,-1.5) to [out=0, in=-120] (pr);
\draw [red, directed] (pr) to node[above, sloped] {$A_1$} (nu);
\draw [red, directed] (pl) to node[above, sloped] {$A_3$}(nu);
\draw [red, directed] (t) to node[right] {$A_2$} (nu);
\draw [red, directed] (pr) to  node[below, sloped] {$A_3^{-t}$}(nd);
\draw [red, directed] (pl) to node[below, sloped] {$A_1^{-t}$}(nd);
\draw [red, directed] (b) to  node[right, pos=0.7] {$A_2^{-t}$} (nd);
\draw[directed] (pr) to node [above] {$D$} (t);
\draw[directed] (pl) to node[below, pos=0.6] {$D^{-1}$} (t);
\draw[directed] (135:0.3) to node[below] {$M_0$} (t);
\draw[blue, directed] (135:0.3) to node[above, pos = 0.8] {$\Lambda$} (135:0);
\draw[directed] circle  [radius=0.3];
\node at (-45:0.4) {$C_0$};
\node [above right] at (t) {\tiny $s$};
\node [below] at (135:0.3) {\tiny $\beta$};

\node [below] at (nu) {\tiny $f_0$};
\node [below] at (nd) {\tiny $f_1$};
\node [below right] at (pr) {\tiny $q_0$};
\node [below] at (pl) {\tiny $q_1$};
\end{tikzpicture}
\end{center}

\caption{ The Stokes' graph $\Sigma$ for the Ugaglia bracket \eqref{UgP}.
}
\label{figSigmaUg}

\end{figure}
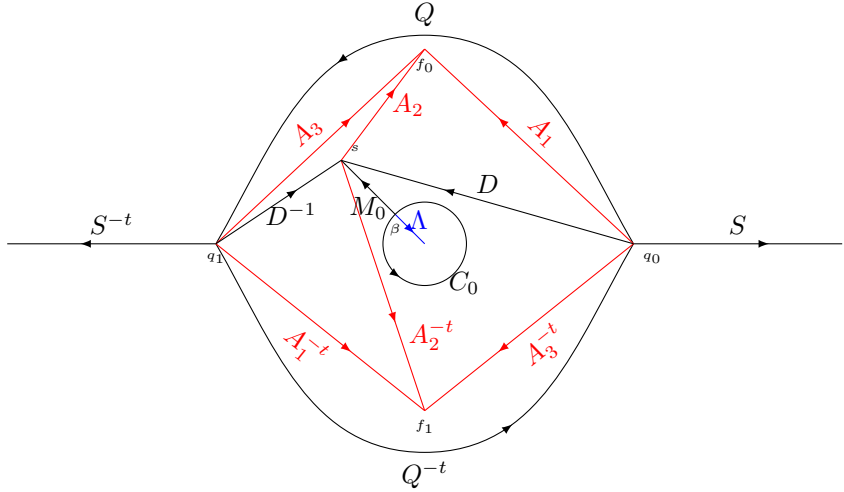
 \paragraph{Jump matrices on the graph $\Sigma$.}
 Let us denote
\be
D = & {\rm diag} \le(z_1,\frac {z_2}{z_1}, \dots, \frac {z_n}{z_{n-1}}, \frac 1 {z_n}\ri)\\
A_1 =& A(\{x_{abc}\}), \ \ A_2 = A(\{x_{bca}\}),\ \ A_3 = A(\{x_{cab}\}),\ \
\\
Q =&  {\rm diag} \le(\sqrt{q_1},\sqrt{\frac {q_2}{q_1}}, \dots, \sqrt{\frac {q_n}{q_{n-1}}}, \frac 1 {\sqrt{q_n}}\ri) P
\ee
and $P_{ab} = (-1)^a \delta_{a,b+n-1}$. The matrix $A(\bs x)$ is described below.
It is convenient to introduce the notations $\alpha_i$, $i=1,\dots, n-1$   for  the simple positive roots of $SL(n)$ and  by  ${\mathrm h}_i$ the dual roots: 
\be
\alpha_i:= {\rm diag}( 0,\dots, \!\!\!\mathop{1}^{i-pos}\!\!\!,-1,0,\dots),\qquad
{\mathrm h}_i := \le(
\begin{array}{cc}
(n-i)\Id_i & 0\\
0& -i\Id_{n-i}
\end{array}
\ri)\;,
\cr
\tr (\alpha_i {\mathrm h}_k) = n \delta_{ik}\;.
\label{rootprods}
\ee
We also denote $M^\star:= PMP^{-1}$ for any matrix $M$.

Following \cite{Bertola:2019ws} we construct the matrices $A_1,A_2,A_3$ depending on $(n-1)(n-2)/2$ variables that we indicize by triple of positive integers adding to $n$: $\bs x = \{x_{abc}, \ a+b+c=n\}$. Let $\mathbb E_{ik}$ be the elementary matrix and define
$$
F_i  = \Id + \mathbb E_{i+1,i}\; ,\ \ \ i=1,\dots, n-1\;,
$$
$$
N_k= \le(\prod_{k\leq  i \leq n-2} 
x_{n-i-1, i-k+1, k}^{-\mathrm h_{i+1}}
F_i\ri) F_{n-1}\; .
$$
Then the matrix $A_1$ is defined as follows \cite{FG}  (here $\sigma  = {\rm diag}(1,-1,1,\dots)$)
\be
A_1({\bs x}) =\sigma\;   \le(\prod_{k=n-1}^1 N_k\ri)\; P\;.
\label{defA1}
\ee
The matrices $A_2$ and $A_3$ are obtained from $A_1$ by cyclically permuting the indices of the variables: 
\be
A_2({\bs x}) = A_1(\{x_{bca}\})\;,\hskip0.7cm
A_3({\bs x}) = A_1(\{x_{cab}\})\;.
\label{defA23}
\ee
The important  property of  the matrices $A_i$ is the equality 
\be
\label{A123}
A_1 A_2 A_3
=\Id\; .
\ee
The anti-diagonal matrix $Q$  is uniquely determined by the requirement that $S$ has ones on the diagonal and it turns out that 
\be
Q = (-1)^{\frac {n+1}2} \prod_{a+b+c=n} x_{abc}^{\frac 1 2(\mathrm h_c^\star + \mathrm h_a)} \prod_{j=1}^{n-1} z_j^{\frac{\alpha_j}2} P.
\ee 
The local monodromy around $z=0$ is 
\be
M_0 = D^{-1} A_2 D A_2^{-t}
\ee
and it is easily seen to be a lower triangular matrix; the diagonal elements of $M_0$ are the diagonal matrix $\Lambda$ and  the eigenvalues of the monodromy. The matrix $C_0$ is the lower--triangular matrix that diagonalizes the local monodromy $M_0$. 
Note that the total monodromy $M$ is 
\be
M = S S^{-t} = Q^{-t} A_3^{-t} M_0 A_3^t Q^t.
\ee
The diagonal elements, $m_1, \dots, m_n$  of  the triangular matrix $M_0$ (eigenvalues of $M$) satisfy the symmetry $m_j =\frac 1{ m_{n+1-j}}$. Correspondingly we impose that the diagonal elements $c_{1}, \dots, c_{n}$ of  $C_0$  satisfy the same symmetry; we call these the {\it toric variables} following similar usage in \cite{Bertola:2019ws}.
The two--form is symplectic because, like in the main body of the paper, it represents the exterior derivative of the Kostant-Kirillov Lie--Poisson symplectic potential. Note that there is no contribution coming from $z=\infty$ because the eigenvector matrix  $G(z) = \Id + Y_1/z + \mathcal O(z^{-2}) $ can be chosen to satisfy $G(-z)= G(z)^{-t}$ so that $Y_1$ is diagonal  free. 
Thus the KK form is simply
\be
\theta = \Tr (\wh L G_0^{-1} \delta G_0) = \Tr ( V \delta G_0 G_0^{-1}), \ \  G_0\in SO(n).
\ee
where $\wh L $  is in the Cartan subalgebra of $so(n)$ and is represented by a skew--symmetric matrix with the same eigenvalues as the diagonal matrix $L$, and $V= G_0 \wh L G_0^{-1}$. 

This extended symplectic structure is related to Ugaglia's Poisson structure as follows: the functions $\mu_j$ generate a torus Hamiltonian action that shifts the log-toric variables $\gamma_j$. Then the Hamiltonian reduction with respect this toric action (i.e. quotienting out the toric variables)  is precisely Ugaglia's bracket. 

\paragraph{Explicit expression in log-canonical coordinates.}
Denote by  $\xi_{abc}, \zeta_j, \mu_j, \gamma_j$  the logarithms of $x_{abc}, z_j, m_j, c_j$, respectively.
The matrix of eigenvalues $\Lambda = {\rm diag}(m_1,\dots, m_n)$ is given by 
\be
\label{evals}
\Lambda = (-1)^{n+1} \prod_{j=1}^{n-1} z_j^{\alpha_j  -  \alpha_j^\star} 
\prod_{a+b+c=n} x_{abc}^{ \mathrm h_b - \mathrm h_b^\star}. 
\ee
To be noted that $m_j = \frac 1{m_{n+1-j}}$.
The Casimirs of the Ugaglia bracket are $m_1,\dots, m_{\lfloor n/2\rfloor}$ (but are not Casimirs of the extended symplectic form).

In the form $\Omega(\Sigma)$ the two vertices with the matrices $A_k$ contribute the same term (See Prop. 4.1 in \cite{Bertola:2019ws})
\be
\omega_f  =\sum_{i + j + k = n\atop 
i' + j' + k' = n}
F_{ijk; i'j'k'} \,\, \d \xi_{ijk}    \wedge \d \xi_{i'j'k'}
\ee
where $F_{ijk; i'j'k'}$ are the following integers
\be
 & F_{ijk;i'j'l'}=    
(\CM_{i,n-j'} - \CM_{i',n-j}    ) H(\Delta i \Delta j)
  +  (\CM_{j,n-k'} - \CM_{j',n-k}   ) H(\Delta j \Delta k)+
 (\CM_{k,n-i'} -\CM_{k',n-i}   ) H(\Delta k \Delta i)
\cr
&\Delta i= i'-i\;, \hskip0.7cm
\Delta j= j'-j\;,  \hskip0.7cm \Delta k= k'-k\;, \ \ 
\CM_{jk}:= \tr(\mathrm h_j \mathrm h_k)=n^2\left({\rm min}(j,k)-\frac {jk}{n}\right)
\label{defCM}
\ee
The contribution of each of the other vertices is straightforward because in each of the summands in \eqref{eq:2formstandard} the matrix one--forms being wedged are of the same triangularity and hence only the diagonal parts matter.
Then we have 
\be
\label{omegaUgaglia} 
\Omega(\Sigma) = 2\omega_f + 2 \omega_{q_0} + \omega_\beta + \omega_s
\ee
where the subscript indicates the contribution due to each of the vertices of $\Sigma$ in the graph shown in Fig. \ref{figSigmaUg}. They are:
\be
\omega_{q_0} =& \sum_{j=1}^{n-1} \sum_{a+b+c=n} \tr\Big(\alpha_j(\mathrm h_a + \mathrm h_c^\star) \Big)\d\zeta_j \wedge \d \xi_{abc}
+\cr
&+\frac{1} 2
\sum_{a+b+c=n\atop a'+b'+c'=n}
\tr \Big(
\mathrm h_c^\star \mathrm h_{a'}-\mathrm h_a \mathrm h_{c'}^\star
\Big)
\d\xi_{abc}\wedge \d \xi_{a'b'c'}
\\
\omega_\beta =& 4\sum_{j=1}^{\lfloor n/2 \rfloor} \d \gamma_j \wedge \d\mu_j
\\
\omega_s =&  2 \sum_{j=1}^{n-1} \sum_{a+b+c=n} \tr \Big(\alpha_j \mathrm h_b^\star \Big) \d \xi_{abc} \wedge \d \zeta_j .
\ee
The factors of $2$ in \eqref{omegaUgaglia}  are due to the fact that, thanks to the symmetry, the vertices $q_0,q_1$ and $f_0,f_1$  give the same contributions.

Denoting by $\{,\}_\Sigma$ the Poisson bracket induced by the form $\Omega(\Sigma)$, we find that 
\be
\{,\}_{\Sigma}  = -8 \{, \}_{_{U}}.
\ee

\section{Conclusion and outlook}

We conclude with a brief discussion of how to generalize the construction of the log--canonical structure  to $SL_n$.
If we consider a polynomial  matrix $A(z)\in sl_n$, then we can still use the same triangulation of the $(2K+2)$ polygon like in Fig. \ref{fig:mod graph form monodromy}. On the perimetric edges the matrices $V$ now are replaced by anti-diagonal matrices (denoted by $S$ in \cite{Bertola:2019ws}, formula (4.4) and following), namely, elements of the Cartan torus times the long permutation in the Weyl group. The only essential difference is that the matrix $A$ which was a constant for the $SL_2$ case, now depends on $(n-1)(n-2)/2$ parameters $x_{abc}$ as in \eqref{defA1}, \eqref{defA23} (but each triangle will have its own set of variables $\bs x$). 

The parameters need to be chosen appropriately so that the Stokes' matrices have unit entries on the diagonal, but the principle is the same as the one we have followed in this paper. 
The details are technical and deferred to the future.

We also remark that the Ugaglia bracket described in Sec. \ref{secUgaglia} is a particular case of the so--called Bondal groupoid \cite{Bondal}; a direct approach expressing these log-canonical coordinates can be found in \cite{Chekhov} but not based directly on the relationship with the Stokes' phenomenon.

Finally we comment on the proper hierarchy for Painlev\'e\ II; this would correspond to $sl_2$ matrices $A(z)$ with a polynomial part of degree $K$ and a simple pole at the origin  satisfying the additional symmetry
\be
A(z) = -\s_1 A(-z) \s_1.
\ee
Handling this situation requires first to extend the construction of the log--canonical coordinate to include the presence of a Fuchsian singularity and hence to generalize the approach to a more general monodromy manifold that extends \eqref{stokesmanifoldintro}.
Second, the symmetry requires a reduction of our description to particular submanifolds of the cluster variety. In particular we would consider only triangulations with an appropriate $\mathbb Z_2$ symmetry. 
These issues are the object of future publications.\\[10pt]

\noindent {\bf Acknowledgements.}  We thank  M. Gekhtman for discussions and comments.  The work of M. B. was supported in part by the Natural Sciences and Engineering Research Council of Canada (NSERC) grant RGPIN-2016-06660. The work of S.T. was supported by the European Union Horizon 2020 research and inovation program under the Marie Sklodowska-Curie RISE 2017 grant agreement no. 778010 IPaDEGAN.

\appendix
\renewcommand{\theequation}{\Alph{section}.\arabic{equation}}
\section {Proof of Prop. \ref{propideal}}
\label{Appideal}
In this appendix we will omit the subscript $_{_{FN}}$ for the Poisson bracket \eqref{eq:PB_K fl newell}. 
The second and third statements follow from the equations \eqref{ideal}.
 The matrix entries of $ F = F_K$ satisfy the homogeneity;
\be
\le\{F_K, \l\ri\}=
\underbrace{\le(\sum_{j=1}^{2K+2}  (-1)^j s_j \frac \pa{\pa s_j} \ri)}_{\mathbb E}F_K =\frac 1 2[F_K,\s_3]
\label{homog}
\ee
which is also the last of the formul\ae. This is easily shown by acting with the group of toric action of conjugations by ${{\rm e}^{\frac \epsilon 2 \s_3}}$, whereas the first equality  in \eqref{homog} follows from the definition of the Poisson bracket \eqref{PB_K}.

For the purpose of the rest of this proof we denote the product of consecutive factors in $F$ as $S_{[a:b]} = S_a\cdots S_b$ with the notation $ U_{2a+1}=S_{2a+1} $ and $ L_{2a} = S_{2a}$.
We have
 \be
\{s_1,F\} =&\le( (1+s_1s_2)\pa_2 +s_1 \sum_{j\geq 3}^{2K+2} (-1)^j \pa_j   
- \frac 1{\l^2} \pa_{2K+2}
- s_1\l\pa_\l\ri) F 
=\nn\\
=&\le(s_1 \mathbb E- s_1\l\pa_\l + s_1^2 \pa_1 + \pa_2 - \frac 1{\l^2} \pa_{2K+2}\ri)F = 
\nn\\
=&
\frac {s_1}  2 [F,\s_3] -   s_1 F\s_3  +
\bigg(\U{s_1}\s_-+s_1^2 \s_+ \L{s_2}  \bigg)
 S_{[3:2K+2]}\l^{\s_3}
-\frac {1} {\l^2} S_{[1:2k+1]} \s_-\l^{\s_3}.
\ee
We now use that $\s_-L_2 = L_2\s_-=\s_-$ and $[\s_-,U_1]=s_1\s_3$ so that we can continue:
\be
\frac {s_1}  2 & [F,\s_3] -   s_1 F\s_3  +
\bigg(\s_-U_1 +\underbrace{s_1\s_3+s_1^2 \s_+ }_{=s_1\s_3U_1}  \bigg)
 S_{[2:2K+2]}\l^{\s_3}
-\underbrace{\frac {1} {\l^2} S_{[1:2K+2]} \s_-\l^{\s_3}}_{=F\s_-}=
\\
=& \frac {s_1}  2 [F,\s_3] -   s_1 F\s_3  + \le[
\begin{array}{cc}
s_1&0\\ 1 &-s_1
\end{array}
\ri]F -F \s_-  = \frac {s_1}2 [\s_3,F] +[ \s_-,F] .
\ee
As for the remaining equations, consider $\{s_j, F\}$:
\be
(-1)^{\ell-1} \{s_\ell, F\} = s_\ell\le (-\sum_{j=1}^{\ell-1} (-1)^{j}  s_j \pa_j 
+ \sum_{j=\ell+1}^{2K+2} (-1)^{j} s_j \pa_j -  \l \pa_\l + (-1)^{\ell} \pa_{\ell+1} -(-1)^\ell \pa_{\ell-1}\ri)F
=\\
=
 s_\ell\le (\mathbb E  
 + (-1)^\ell s_\ell^2\pa_\ell
-2 \sum_{j=1}^{\ell-1} (-1)^{j} s_j \pa_j - \l \pa_\l + (-1)^\ell \pa_{\ell+1} -(-1)^\ell  \pa_{\ell-1}\ri)F
=\\
=
\frac{ s_\ell}2[F,\s_3]   -s_\ell F\s_3
-2s_\ell \underbrace{\sum_{j=1}^{\ell-1} (-1)^{j} s_j \pa_j }_{\mathbb E_{\ell-2}} F +(-1)^\ell  \pa_{\ell+1} F- (-1)^\ell \pa_{\ell-1} F +(-1)^\ell s_\ell^2 \pa_\ell F
\label{23}
\ee
Observe that the vector field $\mathbb E_\ell$ is the generator of the toric action of conjugations by  $\epsilon^{\frac {\s_3}2}$ on the first $\ell-1$ factors, and hence its effect is the commutator with $\frac {\s_3}2$; 
\be
 2\sum_{j=1}^{\ell-1} (-1)^{j} s_j \pa_j  F = [ S_{[1:\ell-1]} ,\s_3] S_{[\ell: 2K+2]} \l^{\s_3}
 \ee
 Thus we have (we assume $\ell$ to be even, for definiteness, the case of $\ell$ odd being handled similarly)
 \be
& \eqref{23} = \frac{ s_\ell}2[F,\s_3]   -s_\ell F\s_3 + s_\ell
[\s_3, S_{[1:\ell-1]}]S_{[\ell:2K+2]} \l^{\s_3} 
- \pa_{\ell+1} F+ \pa_{\ell-1} F +(-1)^\ell s_\ell^2 \pa_\ell F
=\\
&=
-\frac{ s_\ell}2[F,\s_3]   -  s_\ell
 S_{[1:\ell-1]}\s_3 S_{[\ell:2K+2]} \l^{\s_3} 
+
S_{[1:\ell-2]} \le(
\L{s_{\ell-1}} \U{s_\ell} \s_- 
\nn\ri.
\\
&\ \ \ \ \le.-\s_- \U{s_\ell}\L{s_{\ell+1}} 
+ s_{\ell}^2 \L{s_{\ell-1}} \s_+ \L{s_{\ell+1}}
\ri) S_{[\ell+2:2K+2]}\l^{\s_3}
 \ee
 A direct computation shows that the matrix in the bracket can be rewritten as 
 \be
s_\ell \L{s_{\ell-1}}s_3 \U{s_\ell}\L{s_{\ell+1}}
 \ee
 and hence we finally find 
 \be
(-1)^{\ell+1} \{s_\ell, F\} = -\frac{ s_\ell}2[F,\s_3]  ,\ \ 2\leq \ell \leq 2k+1
 \ee
 A similar computation shows that (note that $\s_+U=U\s_+=\s_+$)
 \be
 \{s_{2K+2}, F\}
  =& s_{2K+2}\le (\sum_{j=1}^{2K+1} (-1)^{j}  s_j \pa_j 
 +  \l \pa_\l -  \pa_{2K+1} +\frac 1{\l^2} \pa_{1}\ri)F
=\\
 =& s_{2K+2}\le (\mathbb E 
 +  \l \pa_\l -s_{2K+2}\pa_{2K+2}-  \pa_{2K+1} +\frac 1{\l^2} \pa_{1}\ri)F
=\\
=&\frac{s_{2K+2} }{2}[F,\s_3] + s_{2K+2} F\s_3- s_{2K+2}^2 S_{[1:2K+1]}\s_-\l^{\s_3}
-S_{[1:2K]}\s_+L_{2K+2}\l^{\s_3} + \frac 1{\l^2} \s_+S_{[2:2K+2]}\l^{\s_3}
=\\
=&\frac{s_{2K+2} }{2}[F,\s_3] + s_{2K+2} F\s_3- s_{2K+2}^2 S_{[1:2K+1]}\s_-\l^{\s_3}
-S_{[1:2K+1]}\s_+L_{2K+2}\l^{\s_3} + \frac 1{\l^2} \s_+F
=\\
=&\frac{s_{2K+2} }{2}[F,\s_3] + s_{2K+2} F\s_3
-S_{[1:2K+1]}\underbrace{\bigg(\s_+L_{2K+2} + s_{2K+2}^2 \s_-\bigg)}_{= L_{2K+2}(s_{2K+2}\s_3+\s_+)}\l^{\s_3} + \frac 1{\l^2} \s_+F
=\\
=&\frac{s_{2K+2} }{2}[F,\s_3] + s_{2K+2} F\s_3
-s_{2K+2} F\s_3 -\frac 1{\l^2}  F\s_+ + \frac 1{\l^2} \s_+F=\\
=&\frac{s_{2K+2} }{2}[F,\s_3]  - \frac 1{\l^2}[F, \s_+]
\ee
\QED

\bibliographystyle{abbrv}

\begin{thebibliography}{100}

\bibitem{Adams}
M.~Adams, J.~Harnad, and J.~Hurtubise.
\newblock Darboux coordinates and Liouville-Arnold integration in loop
  algebras.
\newblock {\em Communications in Mathematical Physics}, Jan 1993.


\bibitem{Babelon}
O.~Babelon, D.~Bernard, and M.~Talon.
\newblock {\em Introduction to classical integrable systems}.
\newblock Cambridge Monographs on Mathematical Physics. Cambridge University
  Press, Cambridge, 2003.



\bibitem{BertolaIsoTau}
M.~Bertola.
\newblock The dependence on the monodromy data of the isomonodromic tau
  function.
\newblock {\em Comm. Math. Phys.}, 294(2):539--579, 2010.


\bibitem{BertolaCorrection}
M.~Bertola.
\newblock Correction to: {T}he {D}ependence on the {M}onodromy {D}ata of the
  {I}somonodromic {T}au {F}unction.
\newblock {\em Comm. Math. Phys.}, 381(3):1445--1461, 2021.

\bibitem{Bertola:2019ws}
M. Bertola and D. Korotkin.
\newblock Extended Goldman symplectic structure in Fock-Goncharov coordinates.
\newblock 
To appear in Journal of Differential Geometry (2021).



\bibitem{Boalch1}
P. Boalch.
\newblock Symplectic manifolds and isomonodromic deformations.
\newblock {\em Adv. Math.}, 163(2):137--205, 2001.

\bibitem{smapg}
Philip~P. Boalch.
\newblock Stokes matrices, {P}oisson {L}ie groups and {F}robenius manifolds.
\newblock {\em Invent. Math.}, 146(3):479--506, 2001.

\bibitem{BoalchDuke}
P.~Boalch.
\newblock Quasi-{H}amiltonian geometry of meromorphic connections.
\newblock {\em Duke Math. J.}, 139(2):369--405, 2007.


\bibitem{Bondal}
A.~I. Bondal.
\newblock A symplectic groupoid of triangular bilinear forms and the braid
  group.
\newblock {\em Izv. Ross. Akad. Nauk Ser. Mat.}, 68(4):19--74, 2004.

\bibitem{Chekhov}
L.O. Chekhov and M. Shapiro, 
\newblock Darboux coordinates for symplectic groupoid and cluster algebras,
arXiv:2003.07499v2, 43pp

\bibitem{FN2}
H.~Flaschka and A. C.~Newell.
\newblock Monodromy-and spectrum-preserving deformations {I}.
\newblock {\em Communications in Mathematical Physics}, 76(1):65--116, Jan
  1980.

\bibitem{FlaschkaNewell}
H.~Flaschka and A.~C. Newell.
\newblock The inverse monodromy transform is a canonical transformation.
\newblock In {\em Nonlinear problems: present and future ({L}os {A}lamos,
  {N}.{M}., 1981)}, volume~61 of {\em North-Holland Math. Stud.}, pages 65--89.
  North-Holland, Amsterdam-New York, 1982.

\bibitem{FG} Fock, V., Goncharov A., {\it Moduli spaces of local systems and higher Teichm\"uller theory}, Publications Math\'ematiques de l'Institut des Hautes \'Etudes Scientifiques
 {\bf 103}, Issue 1, pp 1-211 (2006), MR2233852, Zbl 1099.14025.
 
%
%


\bibitem{GSV}
M. Gekhtman, M. Shapiro, and A.  Vainshtein.
\newblock {\em Cluster algebras and {P}oisson geometry}, volume 167 of {\em
  Mathematical Surveys and Monographs}.
\newblock American Mathematical Society, Providence, RI, 2010.

\bibitem{Goldman}
W.~M. Goldman.
\newblock The symplectic nature of fundamental groups of surfaces.
\newblock {\em Adv. in Math.}, 54(2):200--225, 1984.



\bibitem{JMU1}
M.~Jimbo, T.~Miwa, and K.~Ueno.
\newblock Monodromy preserving deformation of linear ordinary differential
  equations with rational coefficients. {I}. {G}eneral theory and {$\tau
  $}-function.
\newblock {\em Phys. D}, 2(2):306--352, 1981.

\bibitem{JMU2}
M.~Jimbo and T.~Miwa.
\newblock Monodromy preserving deformation of linear ordinary differential
  equations with rational coefficients. {II}.
\newblock {\em Phys. D}, 2(3):407--448, 1981.

\bibitem{JMU3}
M.~Jimbo and T.~Miwa.
\newblock Monodromy preserving deformation of linear ordinary differential
  equations with rational coefficients. {III}.
\newblock {\em Phys. D}, 4(1):26--46, 1981/82.

\bibitem{KorSam} D. Korotkin, H. Samtleben, 
\newblock Quantization of Coset Space $\sigma$--Models Coupled to Two-Dimensional Gravity, 
\newblock {\em Communications in Mathematical Physics}, {\bf 190}, Issue 2, pp. 411-457 (1997).

\bibitem{Krichever}
I.~Krichever.
\newblock Vector bundles and lax equations on algebraic curves.
\newblock {\em Communications in Mathematical Physics}, Jan 2002.


\bibitem{KricheverPhong}
I.~Krichever and D.~H. Phong.
\newblock Spin chain models with spectral curves from {M} theory.
\newblock {\em Communications in Mathematical Physics}, 213(3):539--574, 2000.



\bibitem{wildcharacter} J. Martinet,  J.-P. Ramis, Elementary acceleration and multisummability. I, Ann. Inst.
H. Poincar\'e\ Phys. Th\'eor. 54(4) (1991), 331--401.

\bibitem{Nakanishi}
K.~Iwaki and T.~Nakanishi.
\newblock Exact {WKB} analysis and cluster algebras.
\newblock {\em J. Phys. A}, 47(47):474009, 98, 2014.


\bibitem{Okamoto}
K.~Okamoto.
\newblock Polynomial {H}amiltonians associated with {P}ainlev{\'e} equations.
  {I}.
\newblock {\em Proc. Japan Acad. Ser. A Math. Sci.}, 56(6):264--268, 1980.


\bibitem{Ugaglia}
M.~Ugaglia.
\newblock On a {P}oisson structure on the space of {S}tokes matrices.
\newblock {\em Internat. Math. Res. Notices}, (9):473--493, 1999.


\bibitem{Wasow}
W.~Wasow.
\newblock {\em Asymptotic expansions for ordinary differential equations}.
\newblock Dover Publications Inc., New York, 1987.
\newblock Reprint of the 1976 edition.



\end{thebibliography}

\end{document}